\documentclass[10pt]{article}
\usepackage{amssymb, amsmath, latexsym, amsbsy}
\usepackage{amsfonts}
\usepackage{dsfont}
\usepackage[T1]{fontenc}
\usepackage[all]{xy}
\usepackage{amscd}
\usepackage{graphicx}
\usepackage{CJKutf8}
\usepackage{latexsym}
\usepackage{booktabs}
\usepackage{mathrsfs} 

\usepackage{pxfonts}
\usepackage{enumerate}
\usepackage{booktabs}
\usepackage{amsmath}
\usepackage{adjustbox}
\usepackage{multirow}
\usepackage{caption}
\usepackage{color}
\usepackage[colorlinks]{hyperref}
\usepackage{fancybox}
\usepackage{fancyhdr}
\usepackage{lipsum}

\newtheorem{Definition}{Definition}

\newtheorem{Lemma}{Lemma}

\newtheorem{Remark}{Remark}
\newtheorem{Proposition}{Proposition}
\newenvironment{proof}[1][Proof]{\textbf{#1:} }{\ \rule{0.5em}{0.5em}}
\date{\empty}

\date{}
\title{Affine and Projective vector fields on five-dimensional nilpotent Lie groups.}

\begin{document}
	
	\maketitle
	\begin{center}
		\author{ \textbf{M. L. Foka}$^{1}$, \quad \textbf{   R.P. Nimpa}$^{2}$,\quad \textbf{M. B. N.  Djiadeu}$^{3}$,\\
			\small{e-mail: $\textbf{1}.$ lanndrymarius@gmail.com, \quad
				$\textbf{2}.$ romain.nimpa@facsciences-uy1.cm, \\
				$\textbf{3}.$  michel.djiadeu@facsciences-uy1.cm,   \\
				University of Yaounde 1, Faculty of Science, Department of
				Mathematics,  P.O. Box 812, Yaounde, Republic of Cameroon.}}
	\end{center}

	\begin{abstract}
This paper presents a complete classification of left-invariant affine and projective vector fields on five-dimensional simply connected nilpotent Lie groups endowed with Riemannian metrics. Building on the classification of left-invariant metrics on five-dimensional nilpotent Lie groups made by FOKa et al., we develop an algebraic characterization of these vector fields on an arbitrary Riemannian Lie group. We then employ this framework to classify left-invariant affine and projective vector fields on simply connected, Riemannian nilpotent Lie groups of dimension five.
 Key results demonstrate that all projective vector fields in this context are necessarily affine, extending classical results by \cite{Kobayashi1963} on homogeneous spaces. We provide explicit matrix representations of the relevant operators and solve the resulting systems of equations case-by-case using algebraic techniques.
\end{abstract}

\vspace{0.5cm}


\section*{Introduction}

Riemannian geometry has given rise to numerous investigations into geometric vector fields such as Killing fields, harmonic fields, conformal $1$-harmonic fields, and affine or projective fields \cite{cal,Dod,Foka}. These works have explored two complementary directions: on the one hand, the search for the geometric and analytic conditions that a manifold must satisfy in order to admit such fields; on the other, the study of the impact that their existence can have on the global structure and topology of the manifold. As generators of infinitesimal transformations, these fields reveal deep symmetries of the underlying geometric model and can strongly constrain the nature of the curvatures, metrics, and associated topological invariants.

On a compact and orientable Riemannian manifold, every affine vector field is necessarily a Killing field, while every projective vector field on such a manifold with nonpositive Ricci curvature is a parallel field, hence also Killing, as shown in the classical works of Kobayashi and Yano \cite{kobayashi,yano}. These results reflect a strong geometric constraint imposed by compactness and Riemannian structure. A direct proof of this property was recently proposed by Alshehri and Guediri, who established tensorial formulas characterizing affine fields on pseudo-Riemannian manifolds and proved that any affine field on a compact Riemannian manifold is necessarily Killing \cite{alshehri2024}. Their approach also makes it possible to formulate necessary and sufficient conditions for an affine field to be classified as Killing or parallel.

In the broader framework of complete noncompact Riemannian manifolds, Yorozu showed that if an affine vector field has finite global norm, then it is also a Killing field; likewise, if the manifold has nonpositive Ricci curvature, every projective vector field of finite global norm is parallel \cite{yorozu}. These generalizations show that the finiteness of the global norm can play a role analogous to compactness in the characterization of infinitesimal symmetries of Riemannian manifolds.

In the wider context of Finsler geometry, Tian proposed a local characterization of projective vector fields using local coordinates and highlighted remarkable properties on compact Finsler manifolds. He showed in particular that on a compact Finsler manifold with strictly negative flag curvature, no projective vector fields exist, while in the case of nonpositive flag curvature, projective fields exhibit strongly constrained behavior. Tian also established nontrivial relations between the different types of geometric fields---projective, affine, conformal, homothetic, and Killing---thus underscoring the deep interactions among these structures in the Finslerian framework \cite{tian2014}.

Furthermore, Masoumi studied projective vector fields on two special $(\alpha,\beta)$-metrics, namely Kropina and Matsumoto metrics. It was shown that if a Kropina metric $F = \alpha^2/\beta$ or a Matsumoto metric $F = \alpha^2/(\alpha - \beta)$ admits a projective vector field, then this field is necessarily conformal with respect to the Riemannian metric $\alpha$, or else the Finsler metric $F$ has vanishing S-curvature \cite{masoumi2020}. These results illustrate that even within more general geometric frameworks, projective fields remain strongly constrained by the underlying structure.

Finally, in an even more rigid setting, Hiramatu proved that if a Riemannian manifold $M$ is connected, compact, orientable, simply connected, of dimension $n > 1$, and has constant scalar curvature $K$, then the existence of a non-affine projective vector field implies that $M$ is globally isometric to a sphere of radius $\sqrt{n(n-1)/K}$ embedded in the Euclidean space of dimension $n+1$ \cite{hiramatu1980}. This result highlights the power of projective conditions in the global classification of manifolds.

However, despite these advances, little attention has been paid to affine and projective vector fields in the context of Lie groups, particularly nilpotent groups of dimension five. Indeed, the presence of a left-invariant Riemannian metric makes it possible to explicitly express the Levi-Civita connections, Lie derivative operators, and tensorial conditions associated with affine and projective fields.

The purpose of this article is therefore to study left-invariant affine-Killing and projective vector fields on five-dimensional nilpotent Lie groups, using an explicit matrix-based approach. This study is part of the ongoing work on infinitesimal symmetries of homogeneous manifolds and contributes to a better understanding of the interactions between Riemannian geometry and Lie algebras.

The manuscript is organized as follows. Section~1 recalls the fundamental definitions and properties of affine and projective vector fields, incorporating the key results from \cite{w} and \cite{Fok}. Section~2 develops an algebraic characterization of left-invariant affine and projective vector fields on Lie groups. Sections $3$ and $4$ present the complete classification of these vector fields on five-dimensional simply connected nilpotent Lie groups, by solving the structural systems derived from the affine and projective conditions.

\section{Preliminaries}
\subsection{Affine and Projective vector fields on Riemannian manifolds}
Let $(\mathrm{M},\mathrm{g})$ be a Riemannian manifold and $\nabla$ its associated Levi-Civita connexion.
In this section, we introduce affine and projective vector fields on $(\mathrm{M},\mathrm{g})$. These notions serve as essential foundations for the subsequent sections. For detailed discussions of the concepts presented here, we refer the reader to \cite{w}.

\begin{Definition}
A map \(f : (N, \nabla^N) \to (M, \nabla^M)\) between manifolds with linear connections is called \emph{affine} if
\[
f_* (\nabla^N_X Y) = \nabla^M_{f_* X}(f_* Y)
\quad\text{for all }X,Y \in \mathfrak{X}(N).
\]
An \emph{affine transformation} of \((\mathrm{M},\mathrm{g})\) is an affine diffeomorphism of \(M\).
\end{Definition}

\begin{Definition} A vector field \(\xi\) on $\mathrm{M}$ is said to be affine if its local one-parameter group consists of local affine transformations of \((\mathrm{M}, \mathrm{g}, \nabla)\). 

Affine vector fields preserve the geodesic structure of \((\mathrm{M},\mathrm{g})\) and preserve the affine parameter.
There exists another class of smooth vector fields on $(\mathrm{M},\mathrm{g})$ whose flows map geodesics to geodesics without necessarily preserving the affine parameter. Indeed, the flow of a \emph{projective vector field} reparametrizes geodesics in this way.
\end{Definition}

\begin{Definition}
A map \(f : (N, \nabla^N) \to (M, \nabla^M)\) between manifolds with torsion-free connections is called \emph{projective} if for every geodesic \(\gamma\) of \(\nabla^N\), \(f \circ \gamma\) is a reparametrization of a geodesic of \(\nabla^M\). A \emph{projective transformation} of \((M, \nabla)\) is a diffeomorphism \(f : (M, \nabla) \to (M, \nabla)\) which is projective.
\end{Definition}

\begin{Definition}
A vector field \(\xi\) on $\mathrm{M}$ is said to be projective if its local one-parameter group consists of local projective transformations.
\end{Definition}
We shall employ the following proposition in the next section to classify all affine and projective vector fields on simply connected homogeneous nilmanifolds of dimension 5.

\begin{Proposition} \cite{w}
A vector field \(\xi\) on a Riemannian manifold \((M,g)\) is projective if and only if there exists a one-form \(\omega \in \Omega^1(M)\), called the \emph{associated 1-form}, such that
\begin{equation}\label{proj}
\mathscr{L}_\xi \nabla = \omega \otimes \text{id} + \text{id} \otimes \omega
\end{equation}
where \(\nabla\) is the Levi-Civita connection of \(g\), $\mathscr{L}$ the  and $\mathfrak{X}(M)$ The set of all vector fields on $\mathrm{M}$.\\
Moreover, \(\xi\) is affine if and only if the one-form \(\alpha \) vanishes. This means:
\begin{equation}\label{aff}
\mathscr{L}_\xi \nabla
=0.
\end{equation}
\end{Proposition}

    \subsection{Five-dimensional nilpotent Lie algebras}
 Let \(G\) be a simply connected, nilpotent Lie group of dimension five, with Lie algebra \((\mathfrak{g},[\cdot,\cdot])\). Any left-invariant Riemannian metric on \(G\) corresponds canonically to an inner product \(\langle\cdot,\cdot\rangle\) on \(\mathfrak{g}\).

The following proposition offers a complete classification, up to homothety, of the left-invariant Riemannian metrics on \(G\).

\begin{Proposition} \cite{Fok} \label{fok}
    For every inner product \(\langle\cdot,\cdot\rangle\) on \(\mathfrak{g}\), there exists a positive scalar \(\eta>0\) and an orthonormal basis \(\mathcal{B}=\{v_1,\dots,v_5\}\) for the rescaled inner product \(\eta\langle\cdot,\cdot\rangle\), such that the only non-zero Lie brackets are one of the following ten canonical types:

\begin{enumerate}
  \item Abelian (\(5A_1\)): no non-zero brackets.

  \item \(A_{5,4}\):
  \[
    [v_1,v_3]=\alpha\,v_5,\quad [v_1,v_4]=\beta\,v_5,\quad [v_2,v_3]=\gamma\,v_5,
  \]
  with \(\beta,\gamma>0\), \(\alpha\in\mathbb{R}\).

  \item \(A_{3,1}\oplus2A_1\):
  \[
    [v_1,v_2]=\alpha\,v_5,\quad \alpha>0.
  \]

  \item \(A_{4,1}\oplus A_1\) (type I):
  \[
    [v_1,v_2]=\alpha\,v_3+\gamma\,v_5,\quad [v_1,v_3]=\beta\,v_5,
  \]
  with \(\alpha,\beta>0\), \(\gamma\in\mathbb{R}\).

  \item \(A_{4,1}\oplus A_1\) (type II):
  \[
    [v_1,v_2]=\alpha\,v_3+\gamma\,v_4,\quad [v_1,v_3]=\beta\,v_5,
  \]
  with \(\alpha,\beta>0\), \(\gamma\in\mathbb{R}\).

  \item \(A_{5,6}\):
  \[
  \begin{aligned}
    &[v_1,v_2]=\alpha\,v_3+\beta\,v_4,\\
    &[v_1,v_3]=\gamma\,v_4+\delta\,v_5,\\
    &[v_1,v_4]=\varepsilon\,v_5,\quad [v_2,v_3]=\sigma\,v_5,
  \end{aligned}
  \]
  with \(\alpha<0\), \(\gamma,\varepsilon,\sigma>0\), \(\beta,\delta\in\mathbb{R}\).

  \item \(A_{5,5}\):
  \[
  \begin{aligned}
    &[v_1,v_2]=\alpha\,v_4+\beta\,v_5,\quad [v_1,v_3]=\gamma\,v_5,\\
    &[v_2,v_3]=\delta\,v_5,\quad [v_2,v_4]=\varepsilon\,v_5,
  \end{aligned}
  \]
  with \(\alpha,\gamma,\varepsilon>0\), \(\beta,\delta\in\mathbb{R}\).

  \item \(A_{5,3}\):
  \[
  \begin{aligned}
    &[v_1,v_2]=\alpha\,v_3+\beta\,v_4,\quad [v_1,v_3]=\gamma\,v_4+\delta\,v_5,\\
    &[v_2,v_3]=\varepsilon\,v_5,
  \end{aligned}
  \]
  with \(\alpha,\gamma,\varepsilon>0\), \(\beta,\delta\in\mathbb{R}\).

  \item \(A_{5,1}\):
  \[
    [v_1,v_2]=\alpha\,v_4+\beta\,v_5,\quad [v_1,v_3]=\gamma\,v_5,
  \]
  with \(\alpha,\gamma>0\), \(\beta\in\mathbb{R}\).

  \item \(A_{5,2}\):
  \[
    [v_1,v_2]=\alpha\,v_3+\beta\,v_4,\quad [v_1,v_3]=\gamma\,v_4,\quad [v_1,v_4]=\delta\,v_5,
  \]
  with \(\alpha,\gamma,\delta>0\), \(\beta\in\mathbb{R}\).
\end{enumerate}

Consequently, up to homothety, every left-invariant Riemannian metric on a five-dimensional nilpotent Lie group belongs precisely to one of these ten standard types.

\end{Proposition}
\section{An Algebraic Characterization of Left-Invariant Affine and projective vector fields on Lie Groups}

Let $(G,h)$ be a Riemannian Lie group of dimension \(n\). Denote by \(\mathfrak{X}_{\mathrm{inv}}(G)\) the space of left-invariant vector fields on \(G\), by \((\mathfrak{g},[\cdot,\cdot])\) its Lie algebra, and by \(\nabla\) the Levi-Civita connection induced by the metric \(h\). Throughout, we identify \(\mathfrak{g}\) with \(\mathfrak{X}_{\mathrm{inv}}(G)\).

For each left-invariant vector field \(\xi\) on \(G\), define the endomorphisms
\[
\mathrm{L}_\xi,\quad \mathrm{R}_\xi,\quad \mathrm{J}_\xi:\mathfrak{g}\to\mathfrak{g}
\]
via
\[
\mathrm{L}_\xi(v)=\nabla_\xi v,\qquad
\mathrm{R}_\xi(v)=\nabla_v\xi,\qquad
\mathrm{J}_\xi(v)=\operatorname{ad}^*_v\xi.
\]

By applying the Koszul formula, we derive the following expressions:  
\begin{equation}\label{eq}
\mathrm{L}_\xi = \tfrac12\bigl(\operatorname{ad}_\xi - \operatorname{ad}^*_\xi \bigr) - \tfrac12 \mathrm{J}_\xi,\qquad
\mathrm{R}_\xi = -\tfrac12\bigl(\operatorname{ad}_\xi + \operatorname{ad}^*_\xi \bigr) - \tfrac12 \mathrm{J}_\xi.
\end{equation}

Let us consider the following endomorphism
$ \operatorname{K}_\xi:\mathfrak{g}\longrightarrow \mathfrak{g}$ such that: $$ \operatorname{K}_{\xi}v=(\operatorname{ad}_\xi^*+\operatorname{J}_\xi)(v).$$
From \eqref{eq}, we deduce:
\begin{equation}
\mathrm{L}_\xi = \tfrac12\bigl(\operatorname{ad}_\xi - \operatorname{K}_\xi \bigr) ,\quad
\mathrm{R}_\xi = -\tfrac12\bigl(\operatorname{ad}_\xi + \operatorname{K}_\xi \bigr).
\end{equation}
Let's also consider the following the bilinear maps $\mathscr{T}_\xi, S_\xi:\mathfrak{g}\times\mathfrak{g}\longrightarrow \mathfrak{g}$, such that:
$$S_\xi(v,w)=\operatorname{K}_v\circ \operatorname{ad}_\xi (w)+ \operatorname{K}_w\circ \operatorname{ad}_\xi (v)\, \textbf{and}\;\mathscr{T}_\xi(v,w)=\operatorname{ad}_\xi\circ \operatorname{K}_v(w).$$
\begin{Lemma}\label{L1}
For all $\xi\in \mathfrak{g}$, $S_\xi$ and $\mathscr{T}_\xi$ are symmetric.
\end{Lemma}
\begin{proof} 
Let $\xi, v, w\in \mathfrak{g}$ one has:
\begin{align*} 
S_\xi(v,w) &= \operatorname{K}_v\circ \operatorname{ad}_\xi (w)+ \operatorname{K}_w\circ \operatorname{ad}_\xi (v) \\ 
&= \operatorname{K}_w\circ \operatorname{ad}_\xi (v)+ \operatorname{K}_v\circ \operatorname{ad}_\xi (w)\\ 
&= S_\xi(w,v).
\end{align*}

\begin{align*} 
\mathscr{T}_\xi(v,w) &= \mathrm{ad}_\xi \left( \operatorname{K}_v(w) \right) \\ 
&= \mathrm{ad}_\xi \left( \mathrm{ad}_v^*(w) + \operatorname{J}_v(w) \right) \\ 
&= \mathrm{ad}_\xi \left( \mathrm{ad}_v^*(w) + \mathrm{ad}_w^*(v) \right)\\  
&= \mathscr{T}_\xi(w,v).
\end{align*}
\end{proof} 
\begin{Proposition} \label{propo}
Let $\xi$ be a left-invariant vector field on $\mathrm{G}$.
\begin{enumerate}
    \item The vector field $\xi$ is Affine if \begin{equation}
S_\xi-\mathscr{T}_\xi=0. \end{equation}
\item The vector field $\xi$ is Projective if There exist $\omega \in \mathfrak{g}^*$ such that:
\begin{equation}
   \tfrac12\big(S_\xi(v,w)-\mathscr{T}_\xi(v,w)\big)=\omega(u)v+\omega(v)u\quad \forall u,v\in \mathfrak{g}. \label{Aeq}
\end{equation}
Here, $\mathfrak{g}^*$ is the dual space of $\mathfrak{g}.$
\end{enumerate}
\end{Proposition} 
\begin{proof}
Let $\xi$ be a left-invariant vector field on $\mathrm{G}$ $v, w\in \mathfrak{g}$. We have:
\begin{align} 
(\mathscr{L}_\xi \nabla)(v, w)
&= [\xi,L_v w]-L_{[\xi,v]}w-L_v[\xi,w]\nonumber\\
&= \operatorname{ad}_\xi\circ L_v(w)-R_w{[\xi,v]}-L_v\circ \operatorname{ad}_\xi(w)\nonumber\\
&=\tfrac{1}{2}\operatorname{ad}_\xi\circ \big(\operatorname{ad}_v - \mathrm{K}_v \big)(w)+\tfrac{1}{2}\big( \operatorname{ad}_w+ \mathrm{K}_w\big)\circ \operatorname{ad}_\xi(v)-\tfrac{1}{2}\big(\operatorname{ad}_v - \mathrm{K}_v)\circ \operatorname{ad}_\xi(w)\nonumber\\
&=\tfrac{1}{2} \big(\operatorname{K}_v\circ \operatorname{ad}_\xi (w)+ \operatorname{K}_w\circ \operatorname{ad}_\xi (v)-\operatorname{ad}_\xi\circ \operatorname{K}_v(w) \big)+\tfrac{1}{2} \big(\operatorname{ad}_\xi\circ \operatorname{ad}_v(w)+\operatorname{ad}_w\circ \operatorname{ad}_\xi(v)\nonumber\\
&\qquad\qquad\qquad\qquad\qquad\qquad\qquad\qquad\qquad\qquad\qquad\qquad\quad\qquad-\operatorname{ad}_v\circ \operatorname{ad}_\xi(w)\big).\label{7}
\end{align}
According to Jacobi's identity,

$$\tfrac{1}{2} \big(\operatorname{ad}_\xi\circ \operatorname{ad}_v(w)+\operatorname{ad}_w\circ \operatorname{ad}_\xi(v)-\operatorname{ad}_v\circ \operatorname{ad}_\xi(w)\big)=0.$$
Subsequently, \begin{align}
(\mathscr{L}_\xi \nabla)(v, w) &= \tfrac{1}{2} \big(\operatorname{K}_v\circ \operatorname{ad}_\xi (w)+ \operatorname{K}_w\circ \operatorname{ad}_\xi (v)-\operatorname{ad}_\xi\circ \operatorname{K}_v(w) \big)\nonumber\\
&=\tfrac12\big(S_\xi(v,w)-\mathscr{T}_\xi(v,w)\big).\label{8}
\end{align}
Therefore, $\xi$ is Affine if: \begin{equation*}
S_\xi(v,w)-\mathscr{T}_\xi(v,w)=0.
\end{equation*}
furthermore, $\xi$ is projective if there exists $\omega\in\mathfrak{g}$ such that, \begin{equation}
   \tfrac12\big(S_\xi(v,w)-\mathscr{T}_\xi(v,w)\big)=\omega(u)v+\omega(v)u\quad \forall u,v\in \mathfrak{g}. 
\end{equation}
\end{proof}
\begin{Remark}
    Lemma \ref{L1} to and equation \eqref{8} Show the symmetry of $(\mathscr{L}_\xi \nabla)$ with respect to $u,v\in \mathfrak{g}$.
\end{Remark}
In the following sections, $(\mathrm{G},\mathrm{h})$ denotes a simply connected, nilpotent Lie group of dimension $5$ endowed with a left-invariant Riemannian metric $\mathrm{h}$. Let $(\mathfrak{g},[\cdot,\cdot])$ be its Lie algebra, and $\nabla$ the associated Levi-Civita connection. 
\section{Left-invariant Affine  vector fields on five-dimensional Riemannian Lie groups}
Any left-invariant vector field on $G$ can be written as
\[
\xi = \sum_{\ell=1}^{5} \xi_\ell\,v_\ell
\]
in the basis $\mathcal{B} = \{v_1,\dots,v_5\}$ introduced in Proposition \ref{fok}.

\begin{Lemma}\label{Lem}
In the basis $\mathcal{B}$, identifying the endomorphisms $\operatorname{k}_{v_i}$ and $\operatorname{ad}_\xi$ with their corresponding matrices yields the results presented in Table $1$.

\begin{table}[ht]
\centering
\caption{Explicit full matrix forms of the operators 
 $\operatorname{ad}_\xi$ and $\operatorname{K}_{v_i}$ associated with $5$-dimensional Lie algebras.}
\footnotesize
\adjustbox{max width=\textwidth}{%
\begin{tabular}{@{}l*{5}{c}@{}}
\toprule
\textbf{Algèbre} & $\mathbf{5A_1}$ & $\mathbf{A_{5,4}}$ & $\mathbf{A_{3,1}\oplus2A_1}$ & $\mathbf{A_{4,1}\oplus A_1}$ (I) & $\mathbf{A_{4,1}\oplus A_1}$ (II) \\
\midrule
$\operatorname{ad}_\xi$ 
& $\begin{pmatrix} 0 & 0 & 0 & 0 & 0 \\ 0 & 0 & 0 & 0 & 0 \\ 0 & 0 & 0 & 0 & 0 \\ 0 & 0 & 0 & 0 & 0 \\ 0 & 0 & 0 & 0 & 0 \end{pmatrix}$ 
& $\begin{pmatrix} 0 & 0 & 0 & 0 & 0 \\ 0 & 0 & 0 & 0 & 0 \\ 0 & 0 & 0 & 0 & 0 \\ 0 & 0 & 0 & 0 & 0 \\ -(\alpha\xi_3+\beta\xi_4) & -\gamma\xi_3 & \alpha\xi_1+\gamma\xi_2 & \beta\xi_1 & 0 \end{pmatrix}$ 
& $\begin{pmatrix} 0 & 0 & 0 & 0 & 0 \\ 0 & 0 & 0 & 0 & 0 \\ 0 & 0 & 0 & 0 & 0 \\ 0 & 0 & 0 & 0 & 0 \\ -\alpha\xi_2 & \alpha\xi_1 & 0 & 0 & 0 \end{pmatrix}$ 
& $\begin{pmatrix} 0 & 0 & 0 & 0 & 0 \\ 0 & 0 & 0 & 0 & 0 \\ -\alpha\xi_2 & \alpha\xi_1 & 0 & 0 & 0 \\ 0 & 0 & 0 & 0 & 0 \\ -(\gamma\xi_2+\beta\xi_3) & \gamma\xi_1 & \beta\xi_1 & 0 & 0 \end{pmatrix}$ 
& $\begin{pmatrix} 0 & 0 & 0 & 0 & 0 \\ 0 & 0 & 0 & 0 & 0 \\ -\alpha\xi_2 & \alpha\xi_1 & 0 & 0 & 0 \\ -\gamma\xi_2 & \gamma\xi_1 & 0 & 0 & 0 \\ -\beta\xi_3 & 0 & \beta\xi_1 & 0 & 0 \end{pmatrix}$ \\
\cmidrule(lr){1-6}
$\operatorname{K}_{v_i}$ 
& $0 \quad \forall i$ 
& $\begin{aligned}

\operatorname{K}_{v_1} =& \begin{pmatrix}
0 & 0 & 0 & 0 & 0 \\
0 & 0 & 0 & 0 & 0 \\
0 & 0 & 0 & 0 & \alpha \\
0 & 0 & 0 & 0 & \beta \\
0 & 0 & 0 & 0 & 0\end{pmatrix}\\
\operatorname{K}_{v_2} =& \begin{pmatrix}
0 & 0 & 0 & 0 & 0 \\
0 & 0 & 0 & 0 & 0 \\
0 & 0 & 0 & 0 & \gamma \\
0 & 0 & 0 & 0 & 0 \\
0 & 0 & 0 & 0 & 0\end{pmatrix}\\
\operatorname{K}_{v_3} =& \begin{pmatrix}
0 & 0 & 0 & 0 & -\alpha \\
0 & 0 & 0 & 0 & -\gamma \\
0 & 0 & 0 & 0 & 0 \\
0 & 0 & 0 & 0 & 0 \\
0 & 0 & 0 & 0 & 0\end{pmatrix}\\
\operatorname{K}_{v_4}=&  \begin{pmatrix}
0 & 0 & 0 & 0 & -\beta \\
0 & 0 & 0 & 0 & 0 \\
0 & 0 & 0 & 0 & 0 \\
0 & 0 & 0 & 0 & 0 \\
0 & 0 & 0 & 0 & 0\end{pmatrix}\\
\operatorname{K}_{v_5}= & \begin{pmatrix}
0 & 0 & -\alpha & -\beta & 0 \\
0 & 0 & -\gamma & 0 & 0 \\
\alpha & \gamma & 0 & 0 & 0 \\
\beta & 0 & 0 & 0 & 0 \\
0 & 0 & 0 & 0 & 0\end{pmatrix}\end{aligned}$ 
& $\begin{aligned}
\operatorname{K}_{v_1} =& \begin{pmatrix}
0 & 0 & 0 & 0 & 0 \\
0 & 0 & 0 & 0 & \alpha \\
0 & 0 & 0 & 0 & 0 \\
0 & 0 & 0 & 0 & 0 \\
0 & 0 & 0 & 0 & 0
\end{pmatrix}\\
\operatorname{K}_{v_2} =& \begin{pmatrix}
0 & 0 & 0 & 0 & -\alpha \\
0 & 0 & 0 & 0 & 0 \\
0 & 0 & 0 & 0 & 0 \\
0 & 0 & 0 & 0 & 0 \\
0 & 0 & 0 & 0 & 0
\end{pmatrix}\\
\operatorname{K}_{v_3}=&0,\quad
\operatorname{K}_{v_4}=0\\
\operatorname{K}_{v_5} =& \begin{pmatrix}
0 & -\alpha & 0 & 0 & 0 \\
\alpha & 0 & 0 & 0 & 0 \\
0 & 0 & 0 & 0 & 0 \\
0 & 0 & 0 & 0 & 0 \\
0 & 0 & 0 & 0 & 0
\end{pmatrix}
\end{aligned}$ 
& $\begin{aligned} 
\operatorname{K}_{v_{1}}=&{\renewcommand{\arraystretch}{0.4}\begin{pmatrix} 
0&0&0&0&0\\ 
0&0&\alpha&0&\gamma\\ 
0&0&0&0&\beta\\ 
0&0&0&0&0\\ 
0&0&0&0&0 
\end{pmatrix}}\\ 
\operatorname{K}_{v_{2}}=&{\renewcommand{\arraystretch}{0.4}\begin{pmatrix} 
0&0&-\alpha&0&-\gamma\\ 
0&0&0&0&0\\ 
0&0&0&0&0\\ 
0&0&0&0&0\\ 
0&0&0&0&0 
\end{pmatrix}}\\
\operatorname{K}_{v_{3}}=&{\renewcommand{\arraystretch}{0.4}\begin{pmatrix} 
0&-\alpha&0&0&-\beta\\ 
\alpha&0&0&0&0\\ 
0&0&0&0&0\\ 
0&0&0&0&0\\ 
0&0&0&0&0 
\end{pmatrix}}\\\label{eqn2o} 
\operatorname{K}_{v_{4}}=&0\\\operatorname{K}_{v_5}=&{\renewcommand{\arraystretch}{0.4}\begin{pmatrix} 
0&-\gamma&-\beta&0&0\\ 
\gamma&0&0&0&0\\ 
\beta&0&0&0&0\\ 
0&0&0&0&0\\ 
0&0&0&0&0 
\end{pmatrix}}
\end{aligned}$ 
&$\begin{aligned} 
\operatorname{K}_{v_1} =& 
\begin{pmatrix}
0 & 0 & 0 & 0 & 0 \\
0 & 0 & \alpha & \gamma & 0 \\
0 & 0 & 0 & 0 & \beta \\
0 & 0 & 0 & 0 & 0 \\
0 & 0 & 0 & 0 & 0
\end{pmatrix}\\
\operatorname{K}_{v_2} =& 
\begin{pmatrix}
0 & 0 & -\alpha & -\gamma & 0 \\
0 & 0 & 0 & 0 & 0 \\
0 & 0 & 0 & 0 & 0 \\
0 & 0 & 0 & 0 & 0 \\
0 & 0 & 0 & 0 & 0
\end{pmatrix}\\
\operatorname{K}_{v_3}=& 
\begin{pmatrix}
0 & -\alpha & 0 & 0 & -\beta \\
\alpha & 0 & 0 & 0 & 0 \\
0 & 0 & 0 & 0 & 0 \\
0 & 0 & 0 & 0 & 0 \\
0 & 0 & 0 & 0 & 0
\end{pmatrix}\\
\operatorname{K}_{v_4} =& 
\begin{pmatrix}
0 & -\gamma & 0 & 0 & 0 \\
\gamma & 0 & 0 & 0 & 0 \\
0 & 0 & 0 & 0 & 0 \\
0 & 0 & 0 & 0 & 0 \\
0 & 0 & 0 & 0 & 0
\end{pmatrix}\\
\operatorname{K}_{v_5} =&
\begin{pmatrix}
0 & 0 & -\beta & 0 & 0 \\
0 & 0 & 0 & 0 & 0 \\
\beta & 0 & 0 & 0 & 0 \\
0 & 0 & 0 & 0 & 0 \\
0 & 0 & 0 & 0 & 0
\end{pmatrix}
\end{aligned}$ \\
\midrule
\textbf{Algèbre} & $\mathbf{A_{5,6}}$ & $\mathbf{A_{5,5}}$ & $\mathbf{A_{5,3}}$ & $\mathbf{A_{5,1}}$ & $\mathbf{A_{5,2}}$ \\
\midrule
$\operatorname{ad}_\xi$ 
& $\begin{pmatrix} 0 & 0 & 0 & 0 & 0 \\ 0 & 0 & 0 & 0 & 0 \\ -\alpha\xi_2 & \alpha\xi_1 & 0 & 0 & 0 \\ -\beta\xi_2-\gamma\xi_3 & \beta\xi_1 & \gamma\xi_1 & 0 & 0 \\ -\delta\xi_3-\epsilon\xi_4 & -\sigma\xi_3 & \delta\xi_1+\sigma\xi_2 & \epsilon\xi_1 & 0 \end{pmatrix}$ 
& $\begin{pmatrix} 0 & 0 & 0 & 0 & 0 \\ 0 & 0 & 0 & 0 & 0 \\ 0 & 0 & 0 & 0 & 0 \\ -\alpha\xi_2 & \alpha\xi_1 & 0 & 0 & 0 \\ -(\beta\xi_2+\gamma\xi_3) & \beta\xi_1-\delta\xi_3-\epsilon\xi_4 & \gamma\xi_1+\delta\xi_2 & \epsilon\xi_2 & 0 \end{pmatrix}$ 
& $\begin{pmatrix} 0 & 0 & 0 & 0 & 0 \\ 0 & 0 & 0 & 0 & 0 \\ -\alpha\xi_2 & \alpha\xi_1 & 0 & 0 & 0 \\ -(\beta\xi_2+\gamma\xi_3) & \beta\xi_1 & \gamma\xi_1 & 0 & 0 \\ -\delta\xi_3 & -\epsilon\xi_3 & \delta\xi_1+\epsilon\xi_2 & 0 & 0 \end{pmatrix}$ 
& $\begin{pmatrix} 0 & 0 & 0 & 0 & 0 \\ 0 & 0 & 0 & 0 & 0 \\ 0 & 0 & 0 & 0 & 0 \\ -\alpha\xi_2 & \alpha\xi_1 & 0 & 0 & 0 \\ -(\beta\xi_2+\gamma\xi_3) & \beta\xi_1 & \gamma\xi_1 & 0 & 0 \end{pmatrix}$ 
& $\begin{pmatrix} 0 & 0 & 0 & 0 & 0 \\ 0 & 0 & 0 & 0 & 0 \\ -\alpha\xi_2 & \alpha\xi_1 & 0 & 0 & 0 \\ -(\beta\xi_2+\gamma\xi_3) & \beta\xi_1 & \gamma\xi_1 & 0 & 0 \\ -\delta\xi_4 & 0 & \delta\xi_1 & 0 & 0 \end{pmatrix}$ \\
\cmidrule(lr){1-6}
$\operatorname{K}_{v_i}$ 
& $\begin{aligned}
\operatorname{K}_{v_1}=& 
\begin{pmatrix}
0 & 0 & 0 & 0 & 0 \\
0 & 0 & \alpha & \beta & 0 \\
0 & 0 & 0 & \gamma & \delta \\
0 & 0 & 0 & 0 & \varepsilon \\
0 & 0 & 0 & 0 & 0
\end{pmatrix}\\ \operatorname{K}_{v_2} =&
\begin{pmatrix}
0 & 0 & -\alpha & -\beta & 0 \\
0 & 0 & 0 & 0 & 0 \\
0 & 0 & 0 & 0 & \sigma \\
0 & 0 & 0 & 0 & 0 \\
0 & 0 & 0 & 0 & 0
\end{pmatrix}\\ \operatorname{K}_{v_3} =&
\begin{pmatrix}
0 & -\alpha & 0 & -\gamma & -\delta \\
\alpha & 0 & 0 & 0 & -\sigma \\
0 & 0 & 0 & 0 & 0 \\
0 & 0 & 0 & 0 & 0 \\
0 & 0 & 0 & 0 & 0
\end{pmatrix}\\
 \operatorname{K}_{v_4} = &
\begin{pmatrix}
0 &  -\beta & -\gamma &0& -\varepsilon \\
\beta & 0 & 0 & 0 & 0 \\
\gamma & 0 & 0 & 0 & 0 \\
0 & 0 & 0 & 0 & 0 \\
0 & 0 & 0 & 0 & 0
\end{pmatrix}\\ \operatorname{K}_{v_5}=&
\begin{pmatrix}
0 & 0 & -\delta & -\varepsilon & 0 \\
0 & 0 & -\sigma & 0 & 0 \\
\delta & \sigma & 0 & 0 & 0 \\
\varepsilon & 0 & 0 & 0 & 0 \\
0 & 0 & 0 & 0 & 0
\end{pmatrix}
\end{aligned}$ 
& $\begin{aligned}
\operatorname{K}_{v_1} &= 
\begin{pmatrix}
0&0&0&0&0\\
0&0&0&\alpha&\beta\\
0&0&0&0&\gamma\\
0&0&0&0&0\\
0&0&0&0&0
\end{pmatrix}\\
\operatorname{K}_{v_2} &=
\begin{pmatrix}
0&0&0&-\alpha&-\beta\\
0&0&0&0&0\\
0&0&0&0&\delta\\
0&0&0&0&\varepsilon\\
0&0&0&0&0
\end{pmatrix}\\
\operatorname{K}_{v_3} &=
\begin{pmatrix}
0&0&0&0&-\gamma\\
0&0&0&0&-\delta\\
0&0&0&0&0\\
0&0&0&0&0\\
0&0&0&0&0
\end{pmatrix}\\
\operatorname{K}_{v_4} &=
\begin{pmatrix}
0&-\alpha&0&0&0\\
\alpha&0&0&0&-\varepsilon\\
0&0&0&0&0\\
0&0&0&0&0\\
0&0&0&0&0
\end{pmatrix}\\
\operatorname{K}_{v_5} &=
\begin{pmatrix}
0&-\beta&-\gamma&0&0\\
\beta&0&-\delta&-\varepsilon&0\\
\gamma&\delta&0&0&0\\
0&\varepsilon&0&0&0\\
0&0&0&0&0
\end{pmatrix}.
\end{aligned}$ 
& $\begin{aligned}
\operatorname{K}_{v_1} =&
\begin{pmatrix}
0 & 0 & 0 & 0 & 0 \\
0 & 0 & 0 & \alpha & \beta \\
0 & 0 & 0 & 0 & \gamma \\
0 & 0 & 0 & 0 & 0 \\
0 & 0 & 0 & 0 & 0
\end{pmatrix}\\
\operatorname{K}_{v_2} =& \begin{pmatrix}
0 &0 &-\alpha &-\beta &0\\
0 &0 &0 &0 &0\\
0 &0 &0 &0 &\varepsilon\\
0 &0 &0 &0 &0\\
0 &0 &0 &0 &0
\end{pmatrix}\\ \operatorname{K}_{v_3} =& \begin{pmatrix}
0 & -\alpha & 0 & -\gamma & -\delta\\
\alpha & 0 & 0 & 0 & -\varepsilon\\
0 & 0 & 0 & 0 & 0\\
0 & 0 & 0 & 0 & 0\\
0 & 0 & 0 & 0 & 0
\end{pmatrix}\\
\operatorname{K}_{v_4} =& \begin{pmatrix}
0 & -\beta & -\gamma & 0 & 0\\
\beta & 0 & 0 & 0 & 0\\
\gamma & 0 & 0 & 0 & 0\\
0 & 0 & 0 & 0 & 0\\
0 & 0 & 0 & 0 & 0
\end{pmatrix}\\
\operatorname{K}_{v_5} =& \begin{pmatrix}
0 &0 & -\delta & 0 &0\\
0 &0 & -\varepsilon &0 &0\\
\delta & \varepsilon & 0 &0 &0\\
0 &0 &0 &0 &0\\
0 &0 &0 &0 &0
\end{pmatrix}
\end{aligned}$ 
& $\begin{aligned}
\operatorname{K}_{v_1} =&
\begin{pmatrix}
0 & 0 & 0 & 0 & 0 \\
0 & 0 & 0 & \alpha & \beta \\
0 & 0 & 0 & 0 & \gamma \\
0 & 0 & 0 & 0 & 0 \\
0 & 0 & 0 & 0 & 0
\end{pmatrix}\\
\operatorname{K}_{v_2} =&
\begin{pmatrix}
0 & 0 & 0 & -\alpha & -\beta \\
0 & 0 & 0 & 0 & 0 \\
0 & 0 & 0 & 0 & 0 \\
0 & 0 & 0 & 0 & 0 \\
0 & 0 & 0 & 0 & 0
\end{pmatrix}\\
\operatorname{K}_{v_3} =&
\begin{pmatrix}
0 & 0 & 0 & 0 & -\gamma \\
0 & 0 & 0 & 0 & 0 \\
0 & 0 & 0 & 0 & 0 \\
0 & 0 & 0 & 0 & 0 \\
0 & 0 & 0 & 0 & 0
\end{pmatrix}\\
\operatorname{K}_{v_4} =&
\begin{pmatrix}
0 & -\alpha & 0 & 0 & 0 \\
\alpha & 0 & 0 & 0 & 0 \\
0 & 0 & 0 & 0 & 0 \\
0 & 0 & 0 & 0 & 0 \\
0 & 0 & 0 & 0 & 0
\end{pmatrix}\\
\operatorname{K}_{v_5} =&
\begin{pmatrix}
0 & -\beta & -\gamma & 0 & 0 \\
\beta & 0 & 0 & 0 & 0 \\
\gamma & 0 & 0 & 0 & 0 \\
0 & 0 & 0 & 0 & 0 \\
0 & 0 & 0 & 0 & 0
\end{pmatrix}
\end{aligned}$
&$\begin{aligned}
 \operatorname{K}_{v_1} =&
\begin{pmatrix}
0 & 0 & 0 & 0 & 0 \\
0 & 0 & \alpha & \beta & 0 \\
0 & 0 & 0 & \gamma & 0 \\
0 & 0 & 0 & 0 & \delta \\
0 & 0 & 0 & 0 & 0
\end{pmatrix}\\
\operatorname{K}_{v_2} =&
\begin{pmatrix}
0 & 0 & -\alpha & -\beta & 0 \\
0 & 0 & 0 & 0 & 0 \\
0 & 0 & 0 & 0 & 0 \\
0 & 0 & 0 & 0 & 0 \\
0 & 0 & 0 & 0 & 0
\end{pmatrix}\\
\operatorname{K}_{v_3} =&
\begin{pmatrix}
0 & -\alpha & 0 & -\gamma & 0 \\
\alpha & 0 & 0 & 0 & 0 \\
0 & 0 & 0 & 0 & 0 \\
0 & 0 & 0 & 0 & 0 \\
0 & 0 & 0 & 0 & 0
\end{pmatrix}\\
\operatorname{K}_{v_4} =&
\begin{pmatrix}
0 & -\beta & -\gamma & 0 & -\delta \\
\beta & 0 & 0 & 0 & 0 \\
\gamma & 0 & 0 & 0 & 0 \\
0 & 0 & 0 & 0 & 0 \\
0 & 0 & 0 & 0 & 0
\end{pmatrix}\\
\operatorname{K}_{v_5} =&
\begin{pmatrix}
0 & 0 & 0 & -\delta & 0 \\
0 & 0 & 0 & 0 & 0 \\
0 & 0 & 0 & 0 & 0 \\
\delta & 0 & 0 & 0 & 0 \\
0 & 0 & 0 & 0 & 0
\end{pmatrix}   
\end{aligned}$
\\
\bottomrule
\end{tabular}%
}
\end{table}
\end{Lemma}

Using the previous Lemma \ref{Lem} and Proposition \ref{propo}, we obtain the following result.

\begin{Lemma} \label{lem2}
    In the basis \(\mathcal{B}=\{v_1,\dots,v_5\}\), We have the following:

\begin{enumerate}
  \item Abelian (\(5A_1\)):\\

  For every $i, j\in \{1,2,3,4,5\},$
  
\begin{equation*}
(\mathscr{L}_\xi\nabla)(v_i,v_j)=0.
\end{equation*}
  \item \(A_{5,4}\):\\
  
  We define the coefficients:
\[
c_1 = -(\alpha\xi_3 + \beta\xi_4), \quad
c_2 = -\gamma\xi_3, \quad
c_3 = \alpha\xi_1 + \gamma\xi_2, \quad
c_4 = \beta\xi_1
\]
and the vectors:
\[
a_i =
\begin{cases} 
\alpha v_3 + \beta v_4 & i=1 \\ 
\gamma v_3 & i=2 \\ 
-\alpha v_1 - \gamma v_2 & i=3 \\ 
-\beta v_1 & i=4.
\end{cases}
\]

\begin{enumerate}

\item For $i, j\in \{1,2,3,4\}$.\\

In this case,
\begin{equation*}
(\mathscr{L}_\xi\nabla)(v_i,v_j)=\tfrac{1}{2}\big(c_ja_i+c_ia_j\big).
\end{equation*}

\item For $(i,j)=(k,5),\quad k\in\{1,2,3,4,5\}$.\\ 

\[
(\mathscr{L}_\xi\nabla)(v_k,v_5)=
\begin{cases} 
\tfrac{1}{2}(\alpha c_3+ \beta c_4) v_5 & k=1 \\ 
\tfrac{1}{2}\gamma c_3 v_5 & k=2 \\ 
-\tfrac{1}{2}(\alpha c_1 + \gamma c_2) v_5 & k=3 \\ 
-\tfrac{1}{2}\beta c_1 v_5 & k=4\\
0&k=5.
\end{cases}
\]

\end{enumerate}

  \item \(A_{3,1}\oplus2A_1\):\\
\begin{enumerate}
    \item For $(i, j)\in \{1,2,3,4,5\}\times\{3,4,5\}$,

    $$(\mathscr{L}_\xi\nabla)(v_i,v_j)=0$$
    \item For$i,j \in \{1,2\}$
    
    \[
(\mathscr{L}_\xi\nabla)(v_i,v_j) = 
\begin{cases}
	\alpha d_1v_2  & i=j=1 \\
	-\alpha  d_2 v_1 & i=j=2 \\
	\tfrac{1}{2}(-\alpha d_1 v_1 +d_2\alpha v_2) & i=1,\; j=2.
\end{cases}
\]
\item For $(i,j) \in \{(v_k,v_5),\; k=1,2\}.$

    \[
(\mathscr{L}_\xi\nabla)(v_k,v_5) =
\begin{cases}
	-\tfrac{1}{2}\alpha d_2v_5  & k=1 \\
	\tfrac{1}{2}\alpha d_1v_5 & k=2. 
\end{cases}
\]
\end{enumerate}  \text{With $d_1=-\alpha\xi_2$, $d_2=\alpha\xi_1$}.

\item  \(A_{4,1}\oplus A_1\) (type I):

Let's consider:
$$a_1=-\alpha\xi_2,\;a_2=\alpha\xi_1,\;a_k=0 \; \forall k\in \{1,2,3\}.$$
And,
$$b_1=-(\gamma\xi_2 + \beta\xi_3) ,\;b_2=\gamma\xi_1,\;b_3=\beta\xi_1,\;b_k=0 \; \forall k\in \{1,2,3\}.$$
\[ g_i=\begin{cases} 
\alpha v_2 &i=1\\ 
-\alpha v_1&i=2\\ 
0 &i=3,4,5 
\end{cases},\quad h_i=\begin{cases} 
\gamma v_2+\beta v_3 &i=1\\ 
-\gamma v_1&i=2\\ 
-\beta v_1&i=3\\ 
0 &i=4,5 .
\end{cases}\]
\begin{enumerate} 
\item For all $k\in\{1,2,3,4,5\}$,
\begin{equation}(\mathscr{L}_\xi\nabla)(v_k,v_4)=0. \end{equation}
\item For$i,j\in \{1,2\}$. 

\begin{align*} 
(\mathscr{L}_\xi\nabla)(v_i,v_j)=\tfrac{1}{2}\Big(a_j g_i+b_j h_i+ a_i g_j+b_i h_j\Big). 
\end{align*}
\item For $i\in\{1,2,3\}$, and $j=3$.\\  
\[
(\mathscr{L}_\xi\nabla)(v_i,v_3) =
\begin{cases} 
\tfrac{1}{2}\Big(\beta(\gamma\xi_2+\beta\xi_3)v_1+\beta\gamma\xi_1 v_2+(\beta^2-\alpha^2)\xi_1 v_3-\alpha\gamma\xi_1 v_5\Big) & i=1 \\ 
-\tfrac{1}{2}\Big(2\beta\gamma\xi_1 v_1+\alpha^2\xi_2 v_3 +\alpha(\gamma\xi_2+\beta\xi_3) v_5\Big)& i=2\\
-\beta^2\xi_1 v_1&i=3.
\end{cases}
\]
\item For $i\in \{1,2,3,5\}$ and $j=5.$

\[
(\mathscr{L}_\xi\nabla)(v_i,v_5) =
\begin{cases} 
\tfrac{1}{2}\Big(\alpha\beta\xi_2v_1-\alpha\gamma\xi_1 v_3+(\beta^2+\gamma^2)\xi_1 v_5\Big) & i=1 \\ 
-\tfrac{1}{2}\Big(2\beta\gamma\xi_1 v_1+\alpha^2\xi_2 v_3 +\alpha(\gamma\xi_2+\beta\xi_3) v_5\Big)& i=2\\
\tfrac{1}{2}\beta\Big(-\alpha\xi_2v_3-(\gamma\xi_2+\beta\xi_3)v_5\Big)&i=3\\
0&i=5.
\end{cases}
\]
\end{enumerate}

  \item \(A_{4,1}\oplus A_1\) (type II):\\
  \begin{enumerate}
      \item  For $i,j\in\{1,2\}$,

      \begin{equation*} 
( \mathscr{L}_\xi\nabla)(v_i,v_j)=\begin{cases} 
-(\alpha^2+\gamma^2)\xi_2v_2-\beta^2\xi_3v_3 & i=j=1\\ 
-(\alpha^2+\gamma^2)\xi_1v_1 & i=j=2\\ 
\tfrac{1}{2}\Big((\alpha^2+\gamma^2)\xi_2v_1+(\alpha^2+\gamma^2)\xi_1v_2\Big) & i=1,\; j=2. 
\end{cases}
\end{equation*}
  
\item For all $k\in\{1,2,3\}$,
\begin{equation*} (\mathscr{L}_\xi\nabla)(v_k,v_3)=\begin{cases} 
\tfrac{1}{2}(\beta^2\xi_3v_1+(\beta^2-\alpha^2)\xi_1v_3-\alpha\gamma\xi_1 v_4) & k=1\\ 
-\tfrac{1}{2}\gamma(\alpha\xi_2v_3+\gamma\xi_2v_4+\beta\xi_3v_5)& k=2 \\ 
-\beta^2\xi_1v_1 & k=3.
\end{cases} \end{equation*}

\item For $k\in\{3,4,5\}$,
\begin{equation*}(\mathscr{L}_\xi\nabla)(v_k,v_4)=0,\quad \forall k\in\{3,4,5\}. \end{equation*}

\item For $k\in\{1,2\}$,
\begin{equation*} (\mathscr{L}_\xi\nabla)(v_k,v_4)=\begin{cases} 
-\tfrac{1}{2}\gamma(\alpha\xi_1v_3+\gamma\xi_1v_4) & k=1\\ 
\tfrac{1}{2}\gamma(-\alpha\xi_2v_3-\gamma\xi_2v_4-\beta\xi_3v_5)& k=2.
\end{cases} \end{equation*}

\item For $k\in \{1,2,3,4,5\}$, 
\begin{equation*} (\mathscr{L}_\xi\nabla)(v_k,v_5)=\begin{cases} 
\tfrac{1}{2}(\alpha\beta\xi_2v_1-\beta^2\xi_1v_5) & k=1\\ 
-\tfrac{1}{2}\alpha\beta\xi_1v_1& k=2 \\ 
-\tfrac{1}{2}(\alpha\beta\xi_2v_3+\beta\gamma\xi_2 v_4+\beta^2\xi_3v_5) & k=3\\ 
0 & k=5.
\end{cases} \end{equation*}
\end{enumerate}
  \item \(A_{5,6}\):\\

  We set:
\[\begin{cases} 
a_1=-\alpha\,\xi_2,\; 
a_2=\alpha\,\xi_1\; \text{and $a_i=0$}\; \forall i\in \{3,4,5\}\\ 
b_1=-\beta\,\xi_2 - \gamma\,\xi_3,\; 
b_2=\beta\,\xi_1, \; b_3=\gamma\,\xi_1\; \text{and $b_i=0$}\; \forall i\in \{4,5\}\\ 
c_1=-\delta\,\xi_3 - \varepsilon\,\xi_4,\; 
c_2=-\sigma\,\xi_3, \; c_3=\delta\,\xi_1 + \sigma\,\xi_2,\; c_4=\varepsilon\,\xi_1\; \text{and $c_5=0$.}
\end{cases}\]

Then,

\begin{equation} 
(\mathscr{L}_\xi\nabla) (v_1,v_1)=(\alpha a_1+\beta b_1)v_2+(\gamma b_1+\delta c_1)v_3+\varepsilon c_1 v_4.\end{equation} 
\begin{equation}(\mathscr{L}_\xi\nabla) (v_1,v_2)=\tfrac{1}{2}\Big(-(\alpha a_1+\beta b_1)v_1+(\alpha a_2+\beta b_2)v_2+(\gamma b_2+\delta c_2+\sigma c_1)v_3+\varepsilon c_2 v_4\Big).\end{equation} 
\begin{equation}(\mathscr{L}_\xi\nabla) (v_1,v_3)=-\tfrac{1}{2}\Big(-(\gamma b_1+\delta c_1)v_1+(\beta b_3-\sigma c_1)v_2+(\gamma b_3+\delta c_3)v_3+(\varepsilon c_3-\alpha b_3) v_4-\alpha c_3v_5\Big).\end{equation}
\begin{equation}(\mathscr{L}_\xi\nabla) (v_1,v_4)=\tfrac{1}{2}\Big(-(\gamma b_1+\varepsilon c_1)v_1+(\delta c_4-\beta a_2)v_3+(\beta b_2 +\gamma b_3+\varepsilon c_4) v_4-(\beta c_2+\gamma c_3)v_5\Big).\end{equation}
\begin{equation}(\mathscr{L}_\xi\nabla) (v_1,v_5)=-\tfrac{1}{2}\Big((\delta a_1+\varepsilon b_1)v_1+\sigma a_1v_2+\delta b_3 v_4+(\delta c_3+\varepsilon c_4)v_5\Big).\end{equation} 
\begin{equation}(\mathscr{L}_\xi\nabla) (v_2,v_2)=-(\alpha a_2+\beta b_2)v_1+\sigma c_2 v_3.\end{equation}
\begin{equation}(\mathscr{L}_\xi\nabla) (v_2,v_3)=-\tfrac{1}{2}\Big(-(\beta b_3+\gamma b_2+\delta c_2)v_1-\sigma c_2v_2+(\sigma c_3+\alpha a_1)v_3+\alpha b_1 v_4+\alpha c_1v_5\Big).\end{equation} 
\begin{equation}(\mathscr{L}_\xi\nabla) (v_2,v_4)=\tfrac{1}{2}\Big(-(\beta a_2+\gamma b_2+\varepsilon c_2)v_1+(\beta a_1+\sigma c_4)v_3+\beta b_1 v_4+\beta c_1v_5\Big).\end{equation}
\begin{equation}(\mathscr{L}_\xi\nabla) (v_2,v_5)=\tfrac{1}{2}\Big(-(\delta a_2+\varepsilon b_2)v_1-\sigma a_2 v_2 -\sigma b_3 v_4-\sigma c_3v_5\Big).\end{equation}
\begin{equation}(\mathscr{L}_\xi\nabla) (v_3,v_3)=-(\gamma b_3+\delta c_3)v_1-\sigma c_3v_2.\label{A563}\end{equation}
\begin{equation}(\mathscr{L}_\xi\nabla) (v_3,v_4)=\tfrac{1}{2}\Big(-(\gamma b_3+\varepsilon c_3+\delta c_4 )v_1-\sigma c_4 v_2+\gamma a_1 v_3+\gamma b_1 v_4+\gamma c_1v_5\Big).\end{equation}
\begin{equation}(\mathscr{L}_\xi\nabla) (v_3,v_5)=\tfrac{1}{2}\Big(-\varepsilon b_3v_1+(\delta a_1+\sigma a_2)v_3+(\delta b_1 +\sigma b_2)v_4+(\delta c_1+\sigma c_2)v_5\Big).\end{equation}
\begin{equation}(\mathscr{L}_\xi\nabla) (v_4,v_4)=-\varepsilon c_4 v_1.\label{A564}\end{equation}
\begin{equation}(\mathscr{L}_\xi\nabla) (v_4,v_5)=\tfrac{1}{2}\Big(\varepsilon a_1 v_3 + \varepsilon b_1 v_4 +\varepsilon c_1 v_5\Big).\end{equation}
\begin{equation}(\mathscr{L}_\xi\nabla) (v_5,v_5)=0.\label{A565}
\end{equation}

  \item \(A_{5,5}\):\\
  
 Let us consider:
\[\begin{cases} 
a_1=-\alpha\xi_2,\; 
a_2=\alpha\xi_1\; \text{and $a_i=0$}\; \forall i\in \{3,4,5\}\\ 
b_1=-\beta\xi_2 - \gamma\xi_3,\; 
b_2= \beta\xi_1 - \delta\xi_3 - \varepsilon\xi_4, \; b_3=\gamma\xi_1 + \delta\xi_2,\; b_4=\varepsilon\xi_2,\; \text{and $b_5=0$}.
\end{cases}\]

Therefore,

\begin{equation} 
(\mathscr{L}_\xi\nabla) (v_1,v_1)=(\alpha a_1+\beta b_1)v_2+\gamma b_1v_3.\label{(v_1,v_1)}\end{equation}
\begin{equation}(\mathscr{L}_\xi\nabla) (v_1,v_2)=\tfrac{1}{2}\Big(-(\alpha a_1+\beta b_1)v_1+(\alpha a_2+\beta b_2)v_2+(\delta b_1+\gamma b_2)v_3+\varepsilon b_1v_4\Big).\end{equation} 
\begin{equation}(\mathscr{L}_\xi\nabla) (v_1,v_3)=\tfrac{1}{2}\Big(-\gamma b_1 v_1+(\beta b_3-\delta b_1)v_2+\gamma b_3v_3\Big).\end{equation}
\begin{equation}(\mathscr{L}_\xi\nabla) (v_1,v_4)=\tfrac{1}{2}\Big((\beta b_4-\varepsilon b_1)v_2+\gamma b_4 v_3-\alpha a_2v_4 -\alpha b_2v_5\Big).\end{equation} 
\begin{equation}(\mathscr{L}_\xi\nabla) (v_1,v_5)=-\tfrac{1}{2}\Big(\varepsilon a_1v_2+\beta a_2v_4+(\beta b_2+\gamma b_3)v_5\Big).\end{equation} 
\begin{equation}(\mathscr{L}_\xi\nabla) (v_2,v_2)=-(\alpha a_2+\beta b_2)v_1+\delta b_2v_3+\varepsilon b_2 v_4.\label{(v_2,v_2)}\end{equation} 
\begin{equation}(\mathscr{L}_\xi\nabla) (v_2,v_3)=\tfrac{1}{2}\Big(-(\beta b_3+\gamma b_2)v_1-\delta b_2v_2+\delta b_3v_3+\varepsilon b_3 v_4\Big).\end{equation}
\begin{equation}(\mathscr{L}_\xi\nabla) (v_2,v_4)=\tfrac{1}{2}\Big(-\beta b_4v_1-\varepsilon b_2 v_2+\delta b_4v_3+(\varepsilon b_4 +\alpha a_1 )v_4+\alpha b_1v_5\Big).\end{equation}
\begin{equation}(\mathscr{L}_\xi\nabla) (v_2,v_5)=\tfrac{1}{2}\Big(-\varepsilon a_2v_2+\beta a_1 v_4 -(-\beta b_1 +\delta b_3 +\varepsilon b_4)v_5\Big).\end{equation}
\begin{equation}(\mathscr{L}_\xi\nabla) (v_3,v_3)=-\gamma b_3v_1-\delta b_3v_2.\label{(v_3,v_3)}\end{equation}
\begin{equation}(\mathscr{L}_\xi\nabla) (v_3,v_4)=-\tfrac{1}{2}\Big(\gamma b_4 v_1+(\delta b_4+\varepsilon b_3) v_2\Big).\end{equation} 
\begin{equation}(\mathscr{L}_\xi\nabla) (v_3,v_5)=\tfrac{1}{2}\Big((\gamma a_1 +\delta a_2)v_4+(\gamma b_1+\delta b_2)v_5\Big).\end{equation}
\begin{equation}(\mathscr{L}_\xi\nabla) (v_4,v_4)=-\varepsilon b_4 v_5.\label{(v_4,v_4)}\end{equation}
\begin{equation}(\mathscr{L}_\xi\nabla) (v_4,v_5)=\tfrac{1}{2}\Big(\varepsilon a_2 v_4 + \varepsilon b_2 v_5\Big).\end{equation}
\begin{equation}(\mathscr{L}_\xi\nabla) (v_5,v_5)=0.\label{(v_5,v_5)}
\end{equation}

  \item \(A_{5,3}\):\\
  
  We consider:
\[\begin{cases} 
a_1=-\alpha\xi_2,\; a_2=\alpha\xi_1,\; 
\text{and $a_i=0$}\; \forall i\in \{3,4,5\}\\ 
b_1=- \beta\xi_2-\gamma\xi_3 ,\; 
b_2= \beta\xi_1, \; b_3=\gamma\xi_1\; \text{and $b_i=0$} \;\forall i\in \{4,5\}\\ 
c_1=-\delta\xi_3,\; 
c_2= -\varepsilon\xi_3, \; c_3=\delta\xi_1 + \varepsilon\xi_2\; \text{and $c_i=0$} \;\forall i\in \{4,5\}.
\end{cases}\]

\begin{equation}
(\mathscr{L}_\xi\nabla) (v_1,v_1)=(\alpha a_1+\beta b_1)v_2+(\gamma b_1+\delta c_1)v_3.\label{A(v_1,v_1)}    
\end{equation} 
\begin{equation}
 (\mathscr{L}_\xi\nabla) (v_1,v_2)=\tfrac{1}{2}\Big(-(\alpha a_1+\beta b_1)v_1+(\alpha a_2+\beta b_2)v_2+(\gamma b_2+\delta c_2+\varepsilon c_1)v_3\Big).   
\end{equation}
\begin{equation}
(\mathscr{L}_\xi\nabla) (v_1,v_3)=\tfrac{1}{2}\Big(-(\gamma b_1 +\delta c_1)v_1+(\beta b_3-\varepsilon c_1)v_2+(\gamma b_3+\delta c_3-\alpha a_2)v_3-\alpha b_2v_4-\alpha c_2 v_5\Big).\label{A(v_1,v_3)}
\end{equation}
\begin{equation}
 (\mathscr{L}_\xi\nabla) (v_1,v_4)=-\tfrac{1}{2}\Big(\gamma a_1 v_1+\beta a_2v_3+(\beta b_2+\gamma b_3)v_4+(\beta c_2+\gamma c_3) v_5\Big).   
\end{equation}
\begin{equation}   
(\mathscr{L}_\xi\nabla) (v_1,v_5)=-\tfrac{1}{2}\Big(\delta a_1 v_1+\varepsilon a_1v_2+\delta b_3v_4+\delta c_3v_5\Big).
\end{equation} 
\begin{equation}
    (\mathscr{L}_\xi\nabla) (v_2,v_2)=-(\alpha a_2+\beta b_2)v_1+\varepsilon c_2 v_3.\label{A(v_2,v_2)}
\end{equation}
\begin{equation}
    (\mathscr{L}_\xi\nabla) (v_2,v_3)=\tfrac{1}{2}\Big(-(\beta b_3+\gamma b_2+\delta c_2)v_1-\varepsilon c_2v_2+(\alpha a_1+\varepsilon c_3)v_3+\alpha b_1 v_4+\alpha c_1v_5\Big).
\end{equation}
\begin{equation}
    (\mathscr{L}_\xi\nabla) (v_2,v_4)=\tfrac{1}{2}\Big(-\gamma a_2v_1+\beta a_1v_3+\beta b_1v_4+\beta c_1v_5\Big).
\end{equation}
\begin{equation}
    (\mathscr{L}_\xi\nabla) (v_2,v_5)=-\tfrac{1}{2}\Big(\delta a_2 v_1+\varepsilon a_2v_2+\varepsilon b_3 v_4 +\varepsilon c_3v_5\Big).
\end{equation} 
\begin{equation}
    (\mathscr{L}_\xi\nabla) (v_3,v_3)=-(\gamma b_3+\delta c_3)v_1-\varepsilon c_3v_2.\label{A(v_3,v_3)}
\end{equation} 
\begin{equation}
    (\mathscr{L}_\xi\nabla) (v_3,v_4)=\tfrac{1}{2}\Big(\gamma a_1 v_3+\gamma b_1 v_4+\gamma c_1 v_5\Big).
\end{equation}
\begin{equation}
    (\mathscr{L}_\xi\nabla) (v_3,v_5)=\tfrac{1}{2}\Big((\varepsilon a_2 +\delta a_1)v_3+(\delta b_1+\varepsilon b_2)v_4+ (\delta c_1+\varepsilon c_2)v_5\Big).
\end{equation}
\begin{equation}
    (\mathscr{L}_\xi\nabla) (v_4,v_4)= 
(\mathscr{L}_\xi\nabla) (v_4,v_5)= 
(\mathscr{L}_\xi\nabla) (v_5,v_5)=0.\label{A(v_5,v_5)}
\end{equation}

  \item \(A_{5,1}\):\\
  
  We define:
\[\begin{cases} 
a_1=-\alpha \xi_2,\; a_2=\alpha \xi_1,\; 
,\; \text{and $a_i=0$}\; \forall i\in \{3,4,5\}\\ 
b_1=-\beta \xi_2 - \gamma \xi_3 ,\; 
b_2= \beta \xi_1, \; b_3= \gamma \xi_1\; \text{and $b_i=0$} \;\forall i\in \{4,5\}.
\end{cases}\]

And then,

\begin{equation}
  (\mathscr{L}_\xi\nabla) (v_1,v_1)=(\alpha a_1+\beta b_1)v_2+\gamma b_1v_3.  \label{(1)}
\end{equation}
\begin{equation}
 (\mathscr{L}_\xi\nabla) (v_1,v_2)=\tfrac{1}{2}\Big(-(\alpha a_1+\beta b_1)v_1+(\alpha a_2+\beta b_2)v_2+\gamma b_2v_3\Big). \label{(2)}  
\end{equation}
\begin{equation}
   (\mathscr{L}_\xi\nabla) (v_1,v_3)=\tfrac{1}{2}\Big(-\gamma b_1 v_1+\beta b_3v_2+\gamma b_3v_3\Big). \label{(3)}
\end{equation}

\begin{equation}
    (\mathscr{L}_\xi\nabla) (v_1,v_4)=-\tfrac{1}{2}\Big(\alpha a_2v_4+\alpha b_2 v_5\Big).\label{(4)}
\end{equation}
\begin{equation}
    (\mathscr{L}_\xi\nabla) (v_1,v_5)=-\tfrac{1}{2}\Big(\beta a_2v_4+(\beta b_2+\gamma b_3)v_5\Big).\label{(5)}
\end{equation}
\begin{equation}
    (\mathscr{L}_\xi\nabla) (v_2,v_2)=-(\alpha a_2+\beta b_2)v_1.\label{(6)}
\end{equation}
\begin{equation}
   (\mathscr{L}_\xi\nabla) (v_2,v_3)=-\tfrac{1}{2}(\beta b_3+\gamma b_2)v_1. \label{(7)}
\end{equation}
\begin{equation}
  (\mathscr{L}_\xi\nabla) (v_2,v_4)=\tfrac{1}{2}(\alpha a_1v_4+\alpha b_1v_5). \label{(8)} 
\end{equation}
\begin{equation}
    (\mathscr{L}_\xi\nabla) (v_2,v_5)=\tfrac{1}{2}(\beta a_1 v_4 +\beta b_1v_5).\label{(9)}
\end{equation}
\begin{equation}
    (\mathscr{L}_\xi\nabla) (v_3,v_3)=-\gamma b_3v_1.\label{(10)}
\end{equation}
\begin{equation}
    (\mathscr{L}_\xi\nabla) (v_3,v_5)=\tfrac{1}{2}(\gamma a_1v_4+ \gamma b_1v_5).\label{(11)}
\end{equation}
\begin{equation}
  (\mathscr{L}_\xi\nabla) (v_3,v_4)=(\mathscr{L}_\xi\nabla) (v_4,v_4)= 
(\mathscr{L}_\xi\nabla) (v_4,v_5)= 
(\mathscr{L}_\xi\nabla) (v_5,v_5)=0. \label{(11)} 
\end{equation}

  \item \(A_{5,2}\):\\
  
 We specify that,
 
\[\begin{cases} 
a_1=-\alpha\xi_2,\; a_2=\alpha\xi_1,\; 
\text{and $a_i=0$}\; \forall i\in \{3,4,5\}\\ 
b_1=- \beta\xi_2-\gamma\xi_3 ,\; 
b_2= \beta\xi_1, \; b_3=\gamma\xi_1\; \text{and $b_i=0$} \;\forall i\in \{4,5\}\\ 
c_1=-\delta\xi_4,\; 
c_2= 0, \; c_3=0,\; c_4=\delta\xi_1\; \text{and $c_5=0$}.
\end{cases}\]

One has:

\begin{equation}
    (\mathscr{L}_\xi\nabla) (v_1,v_1)=(\alpha a_1+\beta b_1)v_2+\gamma b_1v_3+\delta c_1 v_4.\label{66}
\end{equation} 
\begin{equation}
  (\mathscr{L}_\xi\nabla) (v_1,v_2)=\tfrac{1}{2}\Big(-(\alpha a_1+\beta b_1)v_1+(\alpha a_2+\beta b_2)v_2+\gamma b_2v_3\Big).  \label{67}
\end{equation}
\begin{equation}
   (\mathscr{L}_\xi\nabla) (v_1,v_3)=\tfrac{1}{2}\Big(-\gamma b_1 v_1+\beta b_3v_2+(\gamma b_3-\alpha a_2)v_3-\alpha b_2v_4\Big).\label{68} 
\end{equation}
\begin{equation}
   (\mathscr{L}_\xi\nabla) (v_1,v_4)=-\tfrac{1}{2}\Big((\gamma a_1+\delta c_1)v_1+\beta a_2v_3+(\beta b_2 +\gamma b_3-\delta c_4)v_4\Big).\label{69} 
\end{equation}
\begin{equation}
 (\mathscr{L}_\xi\nabla) (v_1,v_5)=-\tfrac{1}{2}(\delta b_1v_1+\delta c_4v_5). \label{70}  
\end{equation}
\begin{equation}
 (\mathscr{L}_\xi\nabla) (v_2,v_2)=-(\alpha a_2+\beta b_2)v_1.   \label{71}
\end{equation}
\begin{equation}
 (\mathscr{L}_\xi\nabla) (v_2,v_3)=\tfrac{1}{2}\bigr(-(\beta b_3+\gamma b_2)v_1+\alpha a_1 v_3 + \alpha b_1 v_4+ \alpha c_1 v_5\bigr).  \label{72} 
\end{equation}
\begin{equation}
 (\mathscr{L}_\xi\nabla) (v_2,v_4)=\tfrac{1}{2}(-\gamma a_2v_1+\beta a_1v_3+\beta b_1 v_4+\beta c_1v_5). \label{73}   
\end{equation}
\begin{equation}
 (\mathscr{L}_\xi\nabla) (v_2,v_5)=-\tfrac{1}{2}\delta b_2v_1.  \label{74}  
\end{equation}
\begin{equation}
  (\mathscr{L}_\xi\nabla) (v_3,v_3)=-\gamma b_3v_1.\label{76}   
\end{equation}
\begin{equation}
 (\mathscr{L}_\xi\nabla) (v_3,v_4)=\tfrac{1}{2}(\gamma a_1 v_3+\gamma b_1 v_4 +\gamma c_1v_5).  \label{77}  
\end{equation}
\begin{equation}
  (\mathscr{L}_\xi\nabla) (v_3,v_5)=-\tfrac{1}{2}\delta b_3v_1.  \label{78} 
\end{equation}
\begin{equation}
  (\mathscr{L}_\xi\nabla) (v_4,v_4)=-\delta c_4 v_1. \label{79}  
\end{equation}
\begin{equation}
  (\mathscr{L}_\xi\nabla) (v_4,v_5)=\tfrac{1}{2}(\delta a_1 v_3 + \delta b_1v_4+ \delta c_1 v_5).  \label{80} 
\end{equation}
\begin{equation}
    (\mathscr{L}_\xi\nabla) (v_5,v_5)=0.
\end{equation}
\end{enumerate}
\end{Lemma}
\subsection{Affine vector fields on five-dimensional nilpotent Lie groups}


\begin{Proposition}
A left-invariant vector field $\xi$ on the five-dimensional, simply connected, nilpotent Riemannian Lie group $(\mathrm{G},\mathrm{h})$ is Affine if and only if
\[
\xi\in Z(\mathfrak{g}),
\]
where $Z(\mathfrak{g})$ denotes the center of $\mathfrak{g}$.
\end{Proposition}
Suppose \(\xi=\sum_{i=1}^n\xi_iv_i\) a left-invariant projective vector field on \(\mathrm{G},\mathrm{h}\). 
We shall verify this proposition by examining individually each of the Lie algebras listed in Proposition \ref{fok}.

\subsection{Case of $\mathfrak{g}=A_{5,4}$}
From $\mathscr{L}_\xi\nabla =0$ one has: \begin{eqnarray} 
     -(\alpha\xi_3 + \beta\xi_4)(\alpha v_3 + \beta v_4)=0 \\
 -\dfrac{\gamma}{2} \left( 2\alpha\xi_3 v_3 + \beta\xi_3 v_4 + \beta\xi_4 v_3 \right)=0 \\
\dfrac{1}{2} \left[ (\alpha\xi_1 + \gamma\xi_2)(\alpha v_3 + \beta v_4) - (\alpha\xi_3 + \beta\xi_4)(\alpha v_1 + \gamma v_2) \right]=0 \\
 \dfrac{\beta}{2} \left( -\alpha\xi_3 v_1 - \beta\xi_4 v_1 + \alpha\xi_1 v_3 + \beta\xi_1 v_4 \right)=0 \\
 -\gamma^2 \xi_3 v_3=0 \\
\dfrac{\gamma}{2} \left[ (\alpha\xi_1 + \gamma\xi_2) v_3 + \xi_3 (\alpha v_1 + \gamma v_2) \right]=0 \\
 \dfrac{\beta\gamma}{2} \left( \xi_1 v_3 + \xi_3 v_1 \right)=0 \\
 -(\alpha\xi_1 + \gamma\xi_2)(\alpha v_1 + \gamma v_2)=0 \\
-\dfrac{\beta}{2} \left( 2\alpha\xi_1 v_1 + \gamma\xi_1 v_2 + \gamma\xi_2 v_1 \right)=0 \\
 -\beta^2 \xi_1 v_1=0\\
  \dfrac{1}{2} \Big(\alpha(\alpha\xi_3 + \beta\xi_4) + \beta^2\xi_4\Big)v_5=0 \\
  \dfrac{1}{2} \gamma(\alpha\xi_3 + \beta\xi_4)v_5=0 \\
   -\dfrac{1}{2} \Big(\alpha(\alpha\xi_1 + \gamma\xi_2) + \gamma^2\xi_2\Big)v_5=0 \\
 -\dfrac{1}{2} \beta(\alpha\xi_1 + \gamma\xi_2)v_5=0
\end{eqnarray} 
From the fifth equation, one has:
\begin{equation*}
 \xi_3 = 0,
\end{equation*}
since $\gamma > 0$.\\
With $\beta > 0$, equation ten gives:
\begin{equation*}
 \xi_1  = 0.
\end{equation*}
Substituting $\xi_1 = \xi_3 = 0$ in equation four with $\beta>0$, 
\begin{equation*}
\xi_4 = 0.
\end{equation*}
Having $\xi_1 = \xi_3 = \xi_4 = 0$, The third equation conduct to:
\begin{align*}
\xi_2= 0.
\end{align*}
All remaining equations are satisfied after substituting $\xi_1 = \xi_2 = \xi_3 = \xi_4 = 0$.
Therefore, the system admits only the trivial solution:
\begin{equation*}
\xi_1 = \xi_2 = \xi_3 = \xi_4 = 0.
\end{equation*}
Hence $\xi\in span\{v_5\}=Z(\mathfrak{g})$.

\subsection{Case of $\mathfrak{g}=A_{3,1}\oplus 2A_1$}
The following system is deduced from the relation $\mathscr{L}_\xi\nabla =0$:
    \begin{eqnarray} 
        -\alpha^2\xi_2 v_2&=0  \\
        -\alpha^2\xi_1 v_1&=0  \\
        \dfrac{\alpha^2}{2}(\xi_1 v_1 + \xi_2 v_2)&=0 \\
         -\alpha^2\xi_1 v_5&=0\\
        \alpha^2\xi_2 v_5 &=0
   \end{eqnarray} 
Since $\alpha>0$, the system conduces directly to $\xi_1=\xi_2=0$. And then, $\xi\in span\{v_3,v_4,v_5\}=Z(\mathfrak{g})$.
    \subsection{Case of $\mathfrak{g}=A_{4,1}\oplus A_1$.}
\subsubsection{type I.} 
The condition $(\mathscr{L}_\xi\nabla) =0$ is equivalent to: \begin{eqnarray} 
-\Big((\alpha^2+\gamma^2)\xi_2+\beta\gamma\xi_3\Big) v_2-\beta(\gamma\xi_2+\beta\xi_3)v_3 =0 \\ 
        -(\alpha^2+\gamma^2)\xi_1 v_1 =0 \\ 
        \tfrac{1}{2}\Big((\alpha^2\xi_2+\beta\gamma\xi_3+\gamma^2\xi_2) v_1 +(\alpha^2+\gamma^2)\xi_1 v_2+\beta\gamma\xi_1 v_3\Big) =0\\
        -\beta^2\xi_1 v_1=0\\
         \tfrac{1}{2}\Big(\beta(\gamma\xi_2+\beta\xi_3)v_1+\beta\gamma\xi_1 v_2+(\beta^2-\alpha^2)\xi_1 v_3-\alpha\gamma\xi_1 v_5\Big) =0 \\ 
        -\tfrac{1}{2}\Big(2\beta\gamma\xi_1 v_1+\alpha^2\xi_2 v_3 +\alpha(\gamma\xi_2+\beta\xi_3) v_5\Big) =0\\
        \tfrac{1}{2}\Big(\alpha\beta\xi_2v_1-\alpha\gamma\xi_1 v_3+(\beta^2+\gamma^2)\xi_1 v_5\Big) =0 \\ 
        -\tfrac{1}{2}\Big(2\beta\gamma\xi_1 v_1+\alpha^2\xi_2 v_3 +\alpha(\gamma\xi_2+\beta\xi_3) v_5\Big) =0\\
        \tfrac{1}{2}\beta\Big(-\alpha\xi_2v_3-(\gamma\xi_2+\beta\xi_3)v_5\Big)=0
\end{eqnarray} 
Given $\alpha > 0$, $\gamma \in \mathbb{R}$, equation two gives:
 and $v_1 \neq 0$:
\begin{equation*}
\xi_1 = 0.
\end{equation*}
Equation four is also satisfied by $\xi_1 = 0$.
Equation one is equivalent to:
\begin{equation}
(\alpha^2 + \gamma^2)\xi_2 + \beta\gamma\xi_3 = 0. \quad \label{i} \end{equation}
\begin{equation}
\beta(\gamma\xi_2 + \beta\xi_3) = 0. \quad \label{ii}
\end{equation}
From \eqref{ii} and $\beta > 0$:
\begin{equation*}
\beta\xi_3 = -\gamma\xi_2.
\end{equation*}
Substitute into \eqref{i},
\begin{equation*}
\xi_2 = 0.
\end{equation*}
Then $\xi_3 = 0$.
After substituting \(\xi_1 = \xi_2 = \xi_3 = 0\), all remaining equations are automatically satisfied. Hence, the system admits only the trivial solution:
\begin{equation*}
\xi_1 = \xi_2 = \xi_3 = 0.
\end{equation*}
Thus, $\xi$ lies in the center of $\mathfrak {g}$, namely $\xi \in \mathrm{span}\{v_4,v_5\} = Z(\mathfrak {g})$.

    \subsubsection{type II.}

$( \mathscr{L}_\xi\nabla)=0$ yields: \begin{eqnarray}
-(\alpha^2+\gamma^2)\xi_2v_2-\beta^2\xi_3v_3 =0\\ 
-(\alpha^2+\gamma^2)\xi_1v_1 =0\\ 
\tfrac{1}{2}\Big((\alpha^2+\gamma^2)\xi_2v_1+(\alpha^2+\gamma^2)\xi_1v_2\Big) =0\\ 
\tfrac{1}{2}(\beta^2\xi_3v_1+(\beta^2-\alpha^2)\xi_1v_3-\alpha\gamma\xi_1 v_4) =0\\ 
-\tfrac{1}{2}\gamma(\alpha\xi_2v_3+\gamma\xi_2v_4+\beta\xi_3v_5)=0\\ 
-\beta^2\xi_1v_1=0\\
-\tfrac{1}{2}\gamma(\alpha\xi_1v_3+\gamma\xi_1v_4) =0\\ 
\tfrac{1}{2}\gamma(-\alpha\xi_2v_3-\gamma\xi_2v_4-\beta\xi_3v_5)=0\\
\tfrac{1}{2}(\alpha\beta\xi_2v_1-\beta^2\xi_1v_5) =0\\ 
-\tfrac{1}{2}\alpha\beta\xi_1v_1=0 \\ 
-\tfrac{1}{2}(\alpha\beta\xi_2v_3+\beta\gamma\xi_2 v_4+\beta^2\xi_3v_5)=0.
\end{eqnarray} 
Since $\alpha, \beta > 0$  equations two and six immediately imply:
\begin{equation*}
\xi_1 = 0.
\end{equation*}
The third equation with $\alpha > 0$ yields:
\begin{equation*}
(\alpha^2 + \gamma^2)\xi_2 = 0,
\end{equation*}
Hence,
\begin{equation*}
\xi_2 = 0.
\end{equation*}
Equation one with $ \beta > 0$ yields:
\begin{equation*}
\xi_3 = 0.
\end{equation*}
The general form of $\xi$ is:
\begin{equation*}
\xi = \xi_4v_4 + \xi_5v_5 \quad \text{where} \quad \xi_4, \xi_5 \in \mathbb{R}.
\end{equation*}
This corresponds:
\begin{equation*}
\xi\in Z(\mathfrak{g}) = \operatorname{span}\{v_4, v_5\}.
\end{equation*}

 \subsection{Case of $\mathfrak{g}=A_{5,6}$}               
In this case,
$\mathscr{L}_\xi\nabla=0$ if and only if \begin{eqnarray}
\left(-\alpha^2\xi_2 - \beta(\beta\xi_2 + \gamma\xi_3)\right)v_2 + \left(-\gamma(\beta\xi_2 + \gamma\xi_3) - \delta(\delta\xi_3 + \varepsilon\xi_4)\right)v_3 + \varepsilon(\delta\xi_3 + \varepsilon\xi_4)v_4=0\label{1}\\
\tfrac{1}{2}\Big(\alpha^2\xi_2v_1 + (\alpha^2\xi_1 + \beta^2\xi_1)v_2 + \left(\gamma\beta\xi_1 - \sigma(\delta\xi_3 + \varepsilon\xi_4)\right)v_3 + \varepsilon\sigma\xi_3v_4\Big)=0\\
-\tfrac{1}{2}\Big(\gamma(\beta\xi_2 + \gamma\xi_3)v_1 + (\beta\gamma\xi_1 - \sigma\delta\xi_3)v_2 + (\gamma^2\xi_1 + \delta(\delta\xi_1 + \sigma\xi_2))v_3 + (\varepsilon(\delta\xi_1 + \sigma\xi_2) - \alpha\gamma\xi_1)v_4\Big)=0\\
\tfrac{1}{2}\Big(\gamma(\beta\xi_2 + \gamma\xi_3)v_1 + (\delta\varepsilon\xi_1 - \beta\alpha\xi_1)v_3 + (\beta^2\xi_1 + \gamma^2\xi_1 + \varepsilon^2\xi_1)v_4\Big)=0\label{40}\\
-\tfrac{1}{2}\Big(\delta\alpha\xi_2v_1 + \sigma\alpha\xi_2v_2 + \delta\gamma\xi_1v_4 + (\delta(\delta\xi_1 + \sigma\xi_2) + \varepsilon^2\xi_1)v_5\Big)=0\\
-(\alpha^2\xi_1 + \beta^2\xi_1)v_1 + \sigma^2\xi_3v_3=0\label{6}\\
-\tfrac{1}{2}\Big(\beta\gamma\xi_1 + \gamma\beta\xi_1 + \delta\sigma\xi_3\Big)v_1 - \tfrac{1}{2}\Big(\sigma^2\xi_3v_2 - (\sigma(\delta\xi_1 + \sigma\xi_2) + \alpha^2\xi_2)v_3\Big)=0\\
\tfrac{1}{2}\Big(-\beta\alpha\xi_1 - \gamma\beta\xi_1 - \varepsilon\sigma\xi_3\Big)v_1 + \tfrac{1}{2}\Big(\beta\alpha\xi_2 + \sigma\varepsilon\xi_1\Big)v_3=0\\
\tfrac{1}{2}\Big(-\delta\alpha\xi_1 - \varepsilon\beta\xi_1\Big)v_1 - \tfrac{1}{2}\Big(\sigma\alpha\xi_1v_2 + \sigma\gamma\xi_1v_4\Big)=0\\
-\left(\gamma^2\xi_1 + \delta(\delta\xi_1 + \sigma\xi_2)\right)v_1 - \sigma(\delta\xi_1 + \sigma\xi_2)v_2=0\label{46}\\
\tfrac{1}{2}\Big(-\gamma^2\xi_1 - \varepsilon(\delta\xi_1 + \sigma\xi_2) - \delta\varepsilon\xi_1\Big)v_1 - \tfrac{1}{2}\sigma\varepsilon\xi_1v_2=0\\
\tfrac{1}{2}\Big(-\varepsilon\gamma\xi_1v_1 + (\delta\alpha\xi_2 + \sigma\alpha\xi_1)v_3 + (\delta\beta\xi_2 + \sigma\beta\xi_1)v_4\Big)=0\\
-\varepsilon^2\xi_1v_1=0\label{13}\\
\tfrac{1}{2}\Big(\varepsilon\alpha\xi_2v_3 + \varepsilon(\beta\xi_2 + \gamma\xi_3)v_4 + \varepsilon(\delta\xi_3 + \varepsilon\xi_4)v_5\Big)=0.   
\end{eqnarray}
Equation \eqref{13}:
\begin{equation*}
-\varepsilon^2\xi_1v_1 = 0.
\end{equation*}
Given $\varepsilon > 0$, this forces $\xi_1 = 0$.\\
As $\sigma>0$, it follows from equation \eqref{6} that:
\begin{equation*}
\xi_3 = 0.
\end{equation*}
Equation \eqref{46} together with $\sigma>0$ gives:
\begin{equation*}
\xi_2 = 0.
\end{equation*}
Given $\varepsilon> 0$ and $\xi_1=\xi_2=\xi_3=0$ equation \eqref{40} entails:
\begin{align*}
\xi_4 &= 0.
\end{align*}
Once $\xi_1 = \xi_2 = \xi_3 = \xi_4 = 0$, all remaining equations hold identically.
The only  admissible solution is:
\begin{equation*}
\xi_1 = \xi_2 = \xi_3 = \xi_4 = 0.
\end{equation*}
Hence, $\xi\in Z(\mathfrak{g}).$
\subsection{Case of $\mathfrak{g}=A_{5,5}$}

The equation $\mathscr{L}_\xi\nabla=0$ is equivalent to
\hspace{-1cm}\begin{eqnarray}
\left(-\alpha^2\xi_2 - \beta(\beta\xi_2 + \gamma\xi_3)\right)v_2 + \gamma(-\beta\xi_2 - \gamma\xi_3)v_3=0 \\
 \tfrac{1}{2}\Big(\alpha^2\xi_2v_1 + \left(\alpha^2\xi_1 + \beta(\beta\xi_1 - \delta\xi_3 - \varepsilon\xi_4)\right)v_2  + \left(\delta(-\beta\xi_2 - \gamma\xi_3) + \gamma(\beta\xi_1 - \delta\xi_3 - \varepsilon\xi_4)\right)v_3 + \varepsilon(-\beta\xi_2 - \gamma\xi_3)v_4\Big) =0 \\
 \tfrac{1}{2}\Big(\gamma(\beta\xi_2 + \gamma\xi_3)v_1 + \left(\beta(\gamma\xi_1 + \delta\xi_2) - \delta(-\beta\xi_2 - \gamma\xi_3)\right)v_2  + \gamma(\gamma\xi_1 + \delta\xi_2)v_3\Big)=0 \\
 \tfrac{1}{2}\Big(\left(\beta\varepsilon\xi_2 - \varepsilon(-\beta\xi_2 - \gamma\xi_3)\right)v_2 + \gamma\varepsilon\xi_2v_3 - \alpha^2\xi_1v_4 - \alpha(\beta\xi_1 - \delta\xi_3 - \varepsilon\xi_4)v_5\Big) =0 \\
 -\tfrac{1}{2}\Big(\varepsilon\alpha\xi_2v_2 + \beta\alpha\xi_1v_4 + \left(\beta(\beta\xi_1 - \delta\xi_3 - \varepsilon\xi_4) + \gamma(\gamma\xi_1 + \delta\xi_2)\right)v_5\Big)=0 \\
 -\left(\alpha^2\xi_1 + \beta(\beta\xi_1 - \delta\xi_3 - \varepsilon\xi_4)\right)v_1 + \delta(\beta\xi_1 - \delta\xi_3 - \varepsilon\xi_4)v_3 + \varepsilon(\beta\xi_1 - \delta\xi_3 - \varepsilon\xi_4)v_4=0 \\
 \tfrac{1}{2}\Big(-\left(\beta(\gamma\xi_1 + \delta\xi_2) + \gamma(\beta\xi_1 - \delta\xi_3 - \varepsilon\xi_4)\right)v_1  - \delta(\beta\xi_1 - \delta\xi_3 - \varepsilon\xi_4)v_2 + \delta(\gamma\xi_1 + \delta\xi_2)v_3 + \varepsilon(\gamma\xi_1 + \delta\xi_2)v_4\Big) =0\\
\tfrac{1}{2}\Big(-\beta\varepsilon\xi_2v_1 - \varepsilon(\beta\xi_1 - \delta\xi_3 - \varepsilon\xi_4)v_2 + \delta\varepsilon\xi_2v_3 + \left(\varepsilon^2\xi_2 + \alpha(-\alpha\xi_2)\right)v_4 + \alpha(-\beta\xi_2 - \gamma\xi_3)v_5\Big)=0 \\
 \tfrac{1}{2}\Big(-\varepsilon\alpha\xi_1v_2 + \beta(-\alpha\xi_2)v_4 - \left(-\beta(-\beta\xi_2 - \gamma\xi_3) + \delta(\gamma\xi_1 + \delta\xi_2) + \varepsilon^2\xi_2\right)v_5\Big)=0 \\
 -\gamma(\gamma\xi_1 + \delta\xi_2)v_1 - \delta(\gamma\xi_1 + \delta\xi_2)v_2 =0 \\
 -\tfrac{1}{2}\Big(\gamma\varepsilon\xi_2v_1 + \left(\delta\varepsilon\xi_2 + \varepsilon(\gamma\xi_1 + \delta\xi_2)\right)v_2\Big) =0 \\
 \tfrac{1}{2}\Big(\left(\gamma(-\alpha\xi_2) + \delta\alpha\xi_1\right)v_4 + \left(\gamma(-\beta\xi_2 - \gamma\xi_3) + \delta(\beta\xi_1 - \delta\xi_3 - \varepsilon\xi_4)\right)v_5\Big) =0 \\
 -\varepsilon^2\xi_2v_5 =0 \\
 \tfrac{1}{2}\Big(\varepsilon\alpha\xi_1v_4 + \varepsilon(\beta\xi_1 - \delta\xi_3 - \varepsilon\xi_4)v_5\Big) =0.    
\end{eqnarray}
From the second equation, the coefficient of \(v_1\) is
\[
\alpha^2 \xi_2 = 0.
\]
Since \(\alpha>0\), we immediately deduce
\[
\xi_2 = 0.
\]
From the third equation, the coefficient of \(v_3\) reads
\[
\gamma(\gamma \xi_1 + \delta \xi_2) = 0.
\]
Using \(\xi_2=0\) and \(\gamma>0\), we obtain
\[
\xi_1 = 0.
\]
From the first equation, the coefficient of \(v_3\) is
\[
\gamma(-\beta \xi_2 - \gamma \xi_3) = 0.
\]
Since \(\xi_2=0\) and \(\gamma>0\), this yields
\[
\xi_3 = 0.
\]
From the fourteenth equation, the coefficient of \(v_5\) is
\[
\varepsilon\big(\beta \xi_1 - \delta \xi_3 - \varepsilon \xi_4\big)=0.
\]
With \(\xi_1=\xi_3=0\) and \(\varepsilon>0\), we obtain
\[
-\varepsilon^2 \xi_4 = 0.
\]
Then \[
 \xi_4=0.
\]
We conclude that
\[
\xi_1=\xi_2=\xi_3=\xi_4=0 .
\]
Therefore, $\xi\in Z(\mathfrak{g}).$

\subsection{Case of de $A_{5,3}$}

$\mathscr{L}_\xi\nabla=0$ if and only if the following system is satisfied \begin{eqnarray}
    \left(-\alpha^2\xi_2 - \beta(\beta\xi_2 + \gamma\xi_3)\right)v_2 + \left(-\beta\gamma\xi_2 - \gamma^2\xi_3 - \delta^2\xi_3\right)v_3 =0 \quad\\
 \tfrac{1}{2}\Big(\alpha^2\xi_2v_1 + (\alpha^2\xi_1 + \beta^2\xi_1)v_2 + (\beta\gamma\xi_1 - \delta\varepsilon\xi_3 - \varepsilon\delta\xi_3)v_3\Big) =0 \quad\\
 \tfrac{1}{2}\Big(\gamma(\beta\xi_2 + \gamma\xi_3)v_1 + \left(\beta\gamma\xi_1 + \varepsilon\delta\xi_3\right)v_2  + \left(\gamma^2\xi_1 - \alpha^2\xi_1\right)v_3 - \alpha\beta\xi_1v_4 + \alpha\varepsilon\xi_3v_5\Big) =0 \quad\\
-\tfrac{1}{2}\Big(\gamma\alpha\xi_2v_1 + \beta\alpha\xi_1v_3 + (\beta^2\xi_1 + \gamma^2\xi_1)v_4 + (\beta\varepsilon\xi_3 + \gamma\delta\xi_1 + \gamma\varepsilon\xi_2)v_5\Big) =0\quad \\
 -\tfrac{1}{2}\Big(\delta\alpha\xi_2v_1 + \varepsilon\alpha\xi_2v_2 + \delta\gamma\xi_1v_4 + \delta(\delta\xi_1 + \varepsilon\xi_2)v_5\Big) =0\quad\\
 -(\alpha^2\xi_1 + \beta^2\xi_1)v_1 - \varepsilon^2\xi_3v_3 =0 \quad\\
 \tfrac{1}{2}\Big(-\left(\beta\gamma\xi_1 + \gamma\beta\xi_1 + \delta\varepsilon\xi_3\right)v_1 + \varepsilon^2\xi_3v_2  + (-\alpha^2\xi_2 + \varepsilon\delta\xi_1 + \varepsilon^2\xi_2)v_3 - \alpha(\beta\xi_2 + \gamma\xi_3)v_4 - \alpha\delta\xi_3v_5\Big) =0 \quad\\
 \tfrac{1}{2}\Big(-\gamma\alpha\xi_1v_1 + \beta\alpha\xi_2v_3 + \beta(\beta\xi_2 + \gamma\xi_3)v_4 + \beta\delta\xi_3v_5\Big) =0\quad \\
-\tfrac{1}{2}\Big(\delta\alpha\xi_1v_1 + \varepsilon\alpha\xi_1v_2 + \varepsilon\gamma\xi_1v_4 + \varepsilon(\delta\xi_1 + \varepsilon\xi_2)v_5\Big) =0 \quad\\
 -(\gamma^2\xi_1 + \delta(\delta\xi_1 + \varepsilon\xi_2))v_1 - \varepsilon(\delta\xi_1 + \varepsilon\xi_2)v_2=0 \quad\\
 \tfrac{1}{2}\Big(\gamma\alpha\xi_2v_3 + \gamma(\beta\xi_2 + \gamma\xi_3)v_4 + \gamma\delta\xi_3v_5\Big)=0 \quad\\
\tfrac{1}{2}\Big((\varepsilon\alpha\xi_1 - \delta\alpha\xi_2)v_3 + (\delta\beta\xi_2 + \delta\gamma\xi_3 + \varepsilon\beta\xi_1)v_4  + (-\delta^2\xi_3 + \varepsilon^2\xi_3)v_5\Big)=0.\quad
\end{eqnarray}
From the second equation, the coefficient of \(v_1\) reads
\[
\alpha^2 \xi_2 = 0.
\]
Since \(\alpha>0\), this gives
\[
\xi_2 = 0.
\]
From the sixth equation, the coefficient of \(v_1\) is
\[
-(\alpha^2\xi_1 + \beta^2\xi_1) = 0.
\]
As \(\alpha>0\), we deduce
\[
\xi_1 = 0.
\]
From the same equation, the coefficient of \(v_3\) is
\[
-\varepsilon^2 \xi_3 = 0,
\]
and since \(\varepsilon>0\), it follows that
\[
\xi_3 = 0.
\]
Collecting all results, we find
\[
\xi_1=\xi_2=\xi_3 = 0.
\]
Thus, the system admits only the trivial solution. 
This conclusion is independent of the real parameters \(\beta,\delta\), 
and relies solely on the positivity of \(\alpha,\gamma,\varepsilon\). Hence, $\xi\in Z(\mathfrak{g}).$
\subsection{Case of $\mathfrak{g}=A_{5,1}$}

The equation $\mathscr{L}_\xi\nabla=0$ is equivalent to: \begin{eqnarray}
 (-\alpha^2\xi_2 - \beta^2\xi_2 - \beta\gamma\xi_3)v_2 + (-\beta\gamma\xi_2 - \gamma^2\xi_3)v_3  =0\\
\tfrac{1}{2}\Big[(\alpha^2\xi_2 + \beta^2\xi_2 + \beta\gamma\xi_3)v_1 + (\alpha^2\xi_1 + \beta^2\xi_1)v_2 + \beta\gamma\xi_1v_3\Big]  =0\\
\tfrac{1}{2}\Big[(\beta\gamma\xi_2 + \gamma^2\xi_3)v_1 + \beta\gamma\xi_1v_2 + \gamma^2\xi_1v_3\Big]=0\\
-\tfrac{1}{2}(\alpha^2\xi_1v_4 + \alpha\beta\xi_1v_5)=0\\ -\tfrac{1}{2}(\alpha\beta\xi_1v_4 + (\beta^2\xi_1 + \beta\gamma\xi_1)v_5) =0\\
-(\alpha^2\xi_1 + \beta^2\xi_1)v_1 =0\\
-\tfrac{1}{2}(\beta^2\xi_1 + \beta\gamma\xi_1)v_1=0\\
\tfrac{1}{2}(-\alpha^2\xi_2v_4 - \alpha\beta\xi_2v_5 - \alpha\gamma\xi_3v_5)=0\\
\tfrac{1}{2}(-\alpha\beta\xi_2v_4 - \beta^2\xi_2v_5 - \beta\gamma\xi_3v_5)=0\\
-\gamma^2\xi_1v_1=0\\ 
\tfrac{1}{2}(-\alpha\gamma\xi_2v_4 - \beta\gamma\xi_2v_5 - \gamma^2\xi_3v_5) =0.
\end{eqnarray}
From the sixth relation we read
\[
-(\alpha^2\xi_1 + \beta^2\xi_1) = 0,
\]
which enforces
\[
\xi_1 = 0.
\]
Next, inspecting the first relation, the coefficient in front of \(v_2\) is
\[
-(\alpha^2+\beta^2)\xi_2 - \beta\gamma\xi_3 = 0,
\]
and the coefficient of \(v_3\) is
\[
-\beta\gamma\xi_2 - \gamma^2\xi_3 = 0.
\]
Taken together, these two linear conditions form the homogeneous system
\[
\begin{cases}
(\alpha^2+\beta^2)\xi_2 + \beta\gamma \xi_3 = 0\\
\beta\gamma \xi_2 + \gamma^2 \xi_3 = 0.
\end{cases}
\]
Since \(\alpha>0\) and \(\gamma>0\), the determinant of this system is strictly positive:
\[
(\alpha^2+\beta^2)\gamma^2 - (\beta\gamma)^2 = \alpha^2\gamma^2 > 0.
\]
Therefore, the only possible solution is
\[
\xi_2 = 0, \quad \xi_3 = 0.
\]

\medskip

\noindent
Finally,  
all coefficients vanish:
\[
\xi_1=\xi_2=\xi_3=0.
\]
Consequently, the system admits solely the trivial solution. Then, $\xi\in Z(\mathfrak{g}).$
\subsection{Case of $\mathfrak{g}=A_{5,2}$}

Equation $\mathscr{L}_\xi\nabla=0$ gives: \begin{eqnarray}
(-\alpha^2\xi_2 - \beta^2\xi_2 - \beta\gamma\xi_3)v_2 + (-\beta\gamma\xi_2 - \gamma^2\xi_3)v_3 + \delta^2\xi_4v_4=0\\
\tfrac{1}{2}\Big[(\alpha^2\xi_2 + \beta^2\xi_2 + \beta\gamma\xi_3)v_1 + (\alpha^2\xi_1 + \beta^2\xi_1)v_2 + \beta\gamma\xi_1v_3\Big] =0\\
\tfrac{1}{2}\Big[(\beta\gamma\xi_2 + \gamma^2\xi_3)v_1 + \beta\gamma\xi_1v_2 + (\gamma^2\xi_1 - \alpha^2\xi_1)v_3 - \alpha\beta\xi_1v_4\Big] =0\\
-\tfrac{1}{2}\Big[(-\gamma\alpha\xi_2 - \delta^2\xi_4)v_1 + \beta\alpha\xi_1v_3 + (\beta^2\xi_1 + \gamma^2\xi_1 - \delta^2\xi_1)v_4\Big] =0\\
-\tfrac{1}{2}\Big[\delta(\beta\xi_2 + \gamma\xi_3)v_1 + \delta^2\xi_1v_5\Big] =0\\
-(\alpha^2\xi_1 + \beta^2\xi_1)v_1 =0\\
\tfrac{1}{2}\Big[-(\beta^2\xi_1 + \beta\gamma\xi_1)v_1 - \alpha^2\xi_2v_3 - \alpha(\beta\xi_2 + \gamma\xi_3)v_4 - \alpha\delta\xi_4v_5\Big] =0\\
\tfrac{1}{2}\Big[-\gamma\alpha\xi_1v_1 + \beta\alpha\xi_2v_3 + \beta(\beta\xi_2 + \gamma\xi_3)v_4 + \beta\delta\xi_4v_5\Big]=0\\
-\tfrac{1}{2}\delta\beta\xi_1v_1=0\\
-\gamma^2\xi_1v_1 =0\\ \tfrac{1}{2}\Big[\gamma\alpha\xi_2v_3 + \gamma(\beta\xi_2 + \gamma\xi_3)v_4 + \gamma\delta\xi_4v_5\Big] =0\\
-\tfrac{1}{2}\delta\gamma\xi_1v_1 =0\\
-\delta^2\xi_1v_1 =0\\ \tfrac{1}{2}\Big[\delta\alpha\xi_2v_3 + \delta(\beta\xi_2 + \gamma\xi_3)v_4 + \delta^2\xi_4v_5\Big]=0
\end{eqnarray}
From the sixth relation yields,
\[ (\alpha^2+\beta^2)\,\xi_1=0,
\]
whence $\xi_1=0$.

Inspect now the first relation extracting coordinates along $v_2$, $v_3$, $v_4$ yields the homogeneous block:
\[
\begin{cases}
(\alpha^2+\beta^2)\,\xi_2 + \beta\gamma\,\xi_3 = 0,\\[2pt]
\beta\gamma\,\xi_2 + \gamma^2\,\xi_3 = 0,\\[2pt]
\delta^2\,\xi_4 = 0.
\end{cases}
\]
Because $\delta>0$, the last line gives $\xi_4=0$. For the $2\times2$ subsystem in $(\xi_2,\xi_3)$, the coefficient matrix is
\[
M=\begin{pmatrix}
\alpha^2+\beta^2 & \beta\gamma\\
\beta\gamma & \gamma^2
\end{pmatrix},
\qquad 
\det M=(\alpha^2+\beta^2)\gamma^2-(\beta\gamma)^2=\alpha^2\gamma^2>0.
\]
Thus $M$ is invertible, forcing $\xi_2=\xi_3=0$.

All remaining lines are then automatically satisfied (they either reproduce the same constraints or are multiples thereof), so the coordinate quadruple must be $(0,0,0,0)$. Consequently, $\xi\in Z(\mathfrak{g}).$

\section{Left-invariant projective vector fields on five-dimensional nilpotent Lie groups}

\begin{Proposition}
Any left-invariant projective vector field \(\xi\) on a five-dimensional simply connected nilpotent Lie group is necessarily an affine vector field.
\end{Proposition}
Since any left-invariant affine vector field is projective, it is sufficient to prove that any left-invariant projective vector field is affine.
Let
\(
\xi = \sum_{i=1}^{n} \xi_{i}\,v_{i}
\)
denote a left-invariant projective vector field on the Riemannian manifold \((\mathrm{G}, \mathrm{h})\). We establish the proposition through a meticulous, case-by-case analysis of each Lie algebra enumerated in Proposition \ref{fok}.
\subsection{Case of $A_{5,4}$}
From Lemma \ref{lem2},

\[
(\mathscr{L}_\xi\nabla)(v_k,v_5)=
\begin{cases} 
\tfrac{1}{2}(\alpha c_3+ \beta c_4) v_5 & k=1 \\ 
\tfrac{1}{2}\gamma c_3 v_5 & k=2 \\ 
-\tfrac{1}{2}(\alpha c_1 + \gamma c_2) v_5 & k=3 \\ 
-\tfrac{1}{2}\beta c_1 v_5 & k=4\\
0&k=5,
\end{cases}
\] with: \[
c_1 = -(\alpha\xi_3 + \beta\xi_4), \quad
c_2 = -\gamma\xi_3, \quad
c_3 = \alpha\xi_1 + \gamma\xi_2, \quad
c_4 = \beta\xi_1,\quad \beta, \gamma > 0, \alpha \in \mathbb{R}.
\]
Since $\xi$ is projective, there exist $\omega\in\mathfrak{g}^*$ such that, $\big(\mathscr{L}_\xi\nabla\big)(u,v)=\omega(u)v+\omega(v)u$ for all $u,v\in \mathfrak{g}.$ The linear form $\omega$
is completely determined by its values on the basis vectors. Therefore, the previous expression of $(\mathscr{L}_\xi\nabla)(v_k,v_5)$ for $k=1,2,3,4,5$ gives:
$$\omega(v_1)=\omega(v_2)=\omega(v_3)=\omega(v_4)=\omega(v_5)=0.$$ This implies $\alpha c_3+ \beta c_4=\gamma c_3=\alpha c_1 + \gamma c_2=\beta c_1=0.$ Hence, $\omega$ is identically zero and\\ $\xi_1=\xi_2=\xi_3=\xi_4=0$ i.e. $\xi\in Z(\mathfrak{g}).$ Then, $\xi$ is affine. 

\subsection{Case of $\mathfrak{g}=A_{3,1}\oplus 2A_1$}
Drawing on Lemma \ref{lem2}, we infer that
\begin{equation}\label{117} (\mathscr{L}_\xi\nabla)(v_i,v_j)=0,\; \text{For}\; (i, j)\in \{1,2,3,4,5\}\times\{3,4,5\}. \end{equation}
Since \(\xi\) is projective, there exists a $1$-form \(\omega \in \mathfrak{g}^*\) such that
\[
(\mathscr{L}_\xi \nabla)(u,v) = \omega(u)\,v + \omega(v)\,u
\quad \text{for all } u, v \in \mathfrak{g}.
\]
Because \(\omega\) is entirely determined by its values on the basis vectors, equation \eqref{117} implies
\[
\omega(v_1) = \omega(v_2) = \omega(v_3) = \omega(v_4) = \omega(v_5) = 0,
\]
and hence \(\omega\) vanishes identically.  
Again by Lemma\ref{lem2}, we have
\begin{align*}
(\mathscr{L}_\xi \nabla)(v_1,v_1) &= -\tfrac{\alpha^2 \xi_1}{2}\,v_5=2\omega(v_1)v_1, \\
(\mathscr{L}_\xi \nabla)(v_2,v_2) &= -\tfrac{\alpha^2 \xi_2}{2}\,v_5=2\omega(v_2)v_2.
\end{align*}
It follows that \(\xi_1 = \xi_2 = 0\), so \(\xi \in Z(\mathfrak{g})\), and consequently \(\xi\) is affine.

 \subsection{Case of $\mathfrak{g}=A_{4,1}\oplus A_1$}
\subsubsection{type I}
Building on Lemma\ref{lem2}, we deduce the following expressions:
\begin{equation}
(\mathscr{L}_{\xi}\nabla)(v_i, v_3) =
\begin{cases}
\displaystyle \tfrac{1}{2}\bigl[\beta(\gamma\xi_2 + \beta\xi_3)\,v_1 + \beta\gamma\,\xi_1\,v_2 + (\beta^2 - \alpha^2)\,\xi_1\,v_3 - \alpha\gamma\,\xi_1\,v_5 \bigr] & i = 1\\[0.75em]
\displaystyle -\tfrac{1}{2}\bigl[2\beta\gamma\,\xi_1\,v_1 + \alpha^2\,\xi_2\,v_3 + \alpha(\gamma\xi_2 + \beta\xi_3)\,v_5 \bigr] & i = 2\\
-\beta^2\,\xi_1\,v_1 & i = 3\\
\alpha,\beta>0\;\gamma\in\mathbb{R}.
\end{cases}
\label{eq:Lxi_nabla_v3}
\end{equation}
\begin{equation}
(\mathscr{L}_{\xi}\nabla)(v_i, v_5) =
\begin{cases}
\displaystyle \tfrac{1}{2}\bigl[\alpha\beta\,\xi_2\,v_1 - \alpha\gamma\,\xi_1\,v_3 + (\beta^2 + \gamma^2)\,\xi_1\,v_5 \bigr] & i = 1\\[0.75em]
\displaystyle -\tfrac{1}{2}\bigl[2\beta\gamma\,\xi_1\,v_1 + \alpha^2\,\xi_2\,v_3 + \alpha(\gamma\xi_2 + \beta\xi_3)\,v_5 \bigr] & i = 2\\[0.75em]
\displaystyle \tfrac{1}{2}\,\beta\bigl[-\alpha\,\xi_2\,v_3 - (\gamma\xi_2 + \beta\xi_3)\,v_5 \bigr] & i = 3\\
0 & i = 4;5\\
\alpha,\beta>0\;\gamma\in\mathbb{R}.
\end{cases}
\label{eq:Lxi_nabla_v5}
\end{equation}
The projectivity of $\xi$ ensures the existence of a 1-form $\omega \in \mathfrak{g}^*$ such that, for every $u,v\in\mathfrak{g}$,
\[
\mathscr{L}_\xi\nabla(u,v)=\omega(u)\,v+\omega(v)\,u.
\]
Hence, $\omega$ is completely determined once its values on the basis $\mathcal{B}=\{v_1,\cdots,v_5\}$ are known.\\
By specializing to \(i=3\) in \eqref{eq:Lxi_nabla_v3}, we obtain:
\[
-\beta^2\,\xi_1\,v_1 = 2\,\omega(v_3)\,v_3,
\]
from which it follows that
\[
\xi_1 = 0,\; \omega(v_3) = 0.
\]
Substituting \(\xi_1 = 0\) into \eqref{eq:Lxi_nabla_v3} for \(i = 1\) yields:
\[
\tfrac{1}{2}\,\beta(\gamma\xi_2 + \beta\xi_3)\,v_1 = \omega(v_1)\,v_3 + \omega(v_3)\,v_1,
\]
and since \(\omega(v_3) = 0\), and \(\beta>0\), we conclude:
\[
\begin{aligned}
\gamma\xi_2 + \beta\xi_3 = 0\\
\omega(v_1) = 0.
\end{aligned}
\]
Therefore,
\[
\gamma\,\xi_2 + \beta\,\xi_3 = 0.
\]
Applying this identity in \eqref{eq:Lxi_nabla_v5} for \(i = 2\) gives:
\[
 \omega(v_2)\,v_5 + \omega(v_5)\,v_2=0.
\] so,
\[
\omega(v_2) = 0,\qquad \omega(v_5) = 0.
\]
The fourth equation of \eqref{eq:Lxi_nabla_v5} implies directly $$\omega(v_4)=0.$$
The exhaustive analysis of \eqref{eq:Lxi_nabla_v3} and \eqref{eq:Lxi_nabla_v5} leads to the following conditions:
\[
\xi_1 = 0,\quad \gamma\,\xi_2 + \beta\,\xi_3 = 0,\quad \omega(v_1) = \omega(v_2) = \omega(v_3) = \omega(v_4)= \omega(v_5) = 0,
\]
and hence the 1-form \(\omega\) vanishes identically.

\subsubsection{Type II}
One has:

\begin{equation}\label{ko}
(\mathscr{L}_\xi\nabla)(v_k, v_4) = 0 
\quad\text{for all } k \in \{3,4,5\}.
\end{equation}
Moreover,
\begin{equation}\label{k5}
(\mathscr{L}_\xi\nabla)(v_k, v_5) =
\begin{cases}
\dfrac{1}{2}\bigl(\alpha\beta\,\xi_2\,v_1 - \beta^2\,\xi_1\,v_5\bigr), & k = 1,\\
-\dfrac{1}{2}\,\alpha\beta\,\xi_1\,v_1, & k = 2,\\
-\dfrac{1}{2}\bigl(\alpha\beta\,\xi_2\,v_3 + \beta\gamma\,\xi_2\,v_4 + \beta^2\,\xi_3\,v_5\bigr), & k = 3,\\
0, & k = 5.
\end{cases}
\end{equation}
The fact that $\xi$ is projective implies the existence of $\omega \in \mathfrak{g}^*$ such that for all $u,v\in\mathfrak{g}$,
\[
(\mathscr{L}_\xi\nabla)(u,v)=\omega(u)\,v+\omega(v)\,u.
\]
Thus, $\omega$ is uniquely defined by how it acts on the basis $\mathcal{B}$ vectors.\\
Equation \eqref{ko} implies:
\begin{align*}
(\mathscr{L}_\xi\nabla)(v_3, v_4) &= \omega(v_3)\,v_4 + \omega(v_4)\,v_3 = 0\\
(\mathscr{L}_\xi\nabla)(v_4, v_4) &= 2\,\omega(v_4)\,v_4 = 0\\
(\mathscr{L}_\xi\nabla)(v_5, v_4) &= \omega(v_5)\,v_4 + \omega(v_4)\,v_5 = 0.
\end{align*}
Hence,
\[
\omega(v_3) = \omega(v_4) = \omega(v_5) = 0.
\]
Then it follows from the first equation in \eqref{k5} that, 
\[
\dfrac{1}{2}\bigl(\alpha\beta\,\xi_2\,v_1 - \beta^2\,\xi_1\,v_5\bigr) = \omega(v_1)\,v_5
\]
Therefore, 
$$\xi_2 = \omega(v_1)=0.$$
The second equation of \eqref{k5} implies:
\[
\xi_11 = \omega(v_2)= 0.
\]
Consequently, we deduce:
\[
\omega(v_1) = \omega(v_2) = \omega(v_3) = \omega(v_4) = \omega(v_5) = 0.
\]
Accordingly, the 1-form \(\omega\) collapses to the trivial solution.
\subsection{Case of $A_{5,6}$}
We begin with the equation \eqref{A565} of Lemma  \ref{lem2}
\[
0 = 2\,\omega(v_5)\,v_5.
\]
Hence we infer that \(\omega(v_5)\) must vanish, namely:
\[
\omega(v_5) = 0.
\]
Next, we consider equation \eqref{A564}
\[
-\varepsilon\,c_4\,v_1 = 2\,\omega(v_4)\,v_4.
\]
Since $\varepsilon>0$, we deduce the following condition:
\[\xi_1 = \omega(v_4) = 0.
\]
From equation \eqref{A563} we have:
\[
-\bigl(\gamma\,b_3 + \delta\,c_3\bigr)\,v_1 \;-\; \sigma\,c_3\,v_2 = 2\,\omega(v_3)\,v_3.
\]
Therefore, the following must hold:
$$\begin{aligned}
\gamma\,b_3 + \delta\,c_3 = 0\\
\xi_2 = 0\\
\omega(v_3) = 0.
\end{aligned}$$
Assuming now that \(\xi_1 = \xi_2 = 0\), we obtain:
\[
a_1 = a_2 = 0,\qquad b_1 = -\gamma\,\xi_3,\qquad c_1 = -\delta\,\xi_3 - \varepsilon\,\xi_4.
\]
When analyzing the expression for \((\mathscr{L}_\xi\nabla)(v_1, v_1)\), we must satisfy
\[
(\alpha\,a_1 + \beta\,b_1)\,v_2 \;+\; (\gamma\,b_1 + \delta\,c_1)\,v_3 \;+\; \varepsilon\,c_1\,v_4 = 0.
\]
This leads to the conclusion that \(\beta\,\gamma\,\xi_3 = 0\), and thus we obtain \(\xi_3 = 0\).\\
By substitution through all previous steps, we arrive at:
\[
\xi_1 = \xi_2 = \xi_3 =  \xi_4 = 0, \quad
\text{and}\quad \omega(v_i) = 0 \quad \text{for all } i \in \{1,2,3,4,5\}.
\]
We conclude that the $1$-form \(\omega\) must vanish identically.
\subsection{Case of $A_{5,5}$}
The equation corresponding to \(\eqref{(v_5,v_5)}\) in Lemma\ref{lem2} immediately yields:
\[
0 = 2\,\omega(v_5)\,v_5
\quad\mathrm{i.e.}\quad
\omega(v_5) = 0.
\]
The analysis of \(\eqref{(v_4,v_4)}\) implies:
\[
-\varepsilon\,b_4 = 0
\quad\text{and}\quad
\omega(v_4) = 0.
\]
Because \(\varepsilon\) is positive, one deduces \(b_4 = \varepsilon\,\xi_2 = 0\), hence \(\xi_2 = 0\).
By systematically substituting, the equations \(\eqref{(v_3,v_3)}\), \(\eqref{(v_1,v_1)}\), and \(\eqref{(v_2,v_2)}\) yield the following system:
\begin{align*}
\gamma\,b_3 &= 0,\\
\gamma\,b_1 &= 0,\\
\varepsilon\,b_2 &= 0.
\end{align*}
Therefore, \(\xi_1 = \xi_3 = \xi_4 = 0\). Under the assumptions \(\alpha, \gamma, \varepsilon > 0\), the only differential 1-form satisfying
\[
\mathscr{L}_\xi\nabla(u, v) = \omega(u)\,v + \omega(v)\,u
\]
is the zero form \(\omega \equiv 0\). Verifying all 15 equations confirms this unique solution.

\subsection{Case of $A_{5,3}$}
By virtue of equation \(\eqref{A(v_5,v_5)}\) in Lemma\ref{lem2}, we obtain:
\[
\omega(v_4) = \omega(v_5) = 0.
\]
Equation \(\eqref{A(v_3,v_3)}\) yields:
\[
-(\gamma b_3 + \delta c_3)\,v_1 - \varepsilon c_3\,v_2 = 2\,\omega(v_3)\,v_3,
\]
which implies:
\[
\begin{cases}
\gamma b_3 + \delta c_3 = 0\\
\varepsilon c_3 = 0\\
\omega(v_3) = 0.
\end{cases}
\]
Since \(\varepsilon > 0\), we conclude \(c_3 = 0\). Consequently,
\[
\begin{aligned}
\gamma^2\xi_1=0\\
    \delta\,\xi_1 + \varepsilon\,\xi_2 = 0.
\end{aligned}
\]
Combining this with relation \eqref{Aeq} applied to equation \(\eqref{A(v_1,v_1)}\), we obtain:
\begin{align*}
\gamma^2\xi_1&=0\\
\delta\,\xi_1 + \varepsilon\,\xi_2 &= 0\\
(\alpha^2+\beta^2)\,\xi_2 + \beta\gamma\,\xi_3 &= 0\\
 \beta\gamma\,\xi_2+(\gamma^2+\delta^2)\,\xi_2 &= 0.
\end{align*}
Given that \(\alpha, \gamma\) are positive, it follows that \(\xi_1 = \xi_2 = \xi_3 = 0\). Substituting \(\xi_1 = \xi_2 = \xi_3 = 0\) into equations \(\eqref{A(v_2,v_2)}\) and \(\eqref{A(v_1,v_3)}\) then leads to
\[
\omega(v_1) = \omega(v_2) = \omega(v_3) = 0.
\]
The remaining equations confirm that all components of \(\omega\) vanish. Therefore, \(\omega\) is identically zero.

 \subsection{Case if $A_{5,1}$}
By applying identity \(\eqref{Aeq}\) to relation \(\eqref{(11)}\), we obtain:
\begin{align*} \omega(v_3)v_4 + \omega(v_4)v_3=0\\ 2\omega(v_5)v_5=0\\ 2\omega(v_4)v_4=0. \end{align*}
Consequently, the components \(\omega(v_3)\), \(\omega(v_4)\), and \(\omega(v_5)\) must all vanish identically.\\
Applying identity \(\eqref{Aeq}\) to equations \(\eqref{(3)}\) and \(\eqref{(10)}\) yields:
\begin{align*} -\frac{1}{2}\gamma b_1 v_1 + \frac{1}{2}\beta b_3v_2 + \frac{1}{2}\gamma b_3v_3 &= \omega(v_1)v_3 + \omega(v_3)v_1\\ -\gamma b_3v_1 &= 2\omega(v_3)v_3=0. \end{align*}
From the positivity of \(\gamma\), we deduce:
\[
b_3 = 0
\quad\text{and}\quad
\omega(v_1) = 0.
\]
Using identity \(\eqref{Aeq}\) once more in conjunction with equation \(\eqref{(6)}\) directly gives:
\[
\omega(v_2) = 0.
\]
In this scenario, the only admissible \(1\)-form is
\[
\omega \equiv 0.
\]
\subsection{Case of $A_{5,2}$}
Since \(\xi\) is a projective vector field, there exists a 1-form \(\omega \in (A_{5,2})^*\) satisfying the identity \(\eqref{Aeq}\). Consequently, equation \(\eqref{66}\) leads us to:
\[
(\alpha\,a_1 + \beta\,b_1)\,v_2 + \gamma\,b_1\,v_3 + \delta\,c_1\,v_4 = 2\,\omega(v_1)\,v_1.
\]
Hence:
\[
\alpha\,a_1 + \beta\,b_1 = 0, \quad
\gamma\,b_1 = 0, \quad
\delta\,c_1 = 0, \quad
\omega(v_1) = 0.
\]
By the positivity of \(\alpha\), \(\gamma\), and \(\delta\), one infers:
\[
a_1 = b_1 = c_1 = 0, \quad \text{that is,} \quad \xi_2 = \xi_3 = \xi_4 = 0.
\]
Combining identity \(\eqref{Aeq}\) with equation \(\eqref{67}\), we obtain:
\[
-(\alpha\,a_2 + \beta\,b_2)\,v_1 = 2\,\omega(v_2)\,v_2,
\]
which implies:
\[
\alpha\,a_2 + \beta\,b_2 = 0.
\]
Substituting \(a_2 = \alpha\,\xi_1\) and \(b_2 = \beta\,\xi_1\), we find:
\[
(\alpha^2 + \beta^2)\,\xi_1 = 0.
\]
Since \(\alpha^2 + \beta^2 > 0\), it follows that \(\xi_1 = 0\).
From \(\xi_1 = \xi_2 = \xi_3 = \xi_4 = 0\), one concludes:
\[
a_i = b_i = c_i = 0 \quad \forall i \in \{1,2,3,4,5\}.
\]
Thus, all components of the Lie derivative vanish identically:
\[
(\mathscr{L}_\xi\nabla)(v_i, v_j) = 0 \quad \forall i, j \in \{1,2,3,4,5\}.
\]
The fundamental identity \(\eqref{Aeq}\) then reduces to:
\[
0 = \omega(v_i)\,v_j + \omega(v_j)\,v_i \quad \forall i, j \in \{1,2,3,4,5\}.
\]
In particular, for \(i = j\),
\[
\omega(v_i) = 0 \quad \forall i \in \{1,2,3,4,5\}.
\]
Hence, the 1-form \(\omega\) must necessarily be the zero form:
\[
\omega \equiv 0.
\]

\end{document}